\documentclass[a4paper,11pt, reqno]{amsart}
\usepackage{hyperref}
\usepackage{amsaddr}
\usepackage{geometry}
\usepackage[british]{babel}
\usepackage{wrapfig}
\usepackage[comma, square, numbers]{natbib}
\usepackage{yfonts}
\usepackage[utf8]{inputenc}
\usepackage{amssymb}
 \usepackage[displaymath, mathlines,running]{lineno}
\usepackage[normalem]{ulem}
\usepackage{amsthm}
\usepackage{graphics}
\usepackage{amsmath}
\usepackage{amsthm}
\usepackage{dirtytalk}
\usepackage{amstext}
\usepackage{subfigure}
\usepackage{engrec}
\usepackage{floatflt}
\usepackage{rotating}
\usepackage[safe,extra]{tipa}
\usepackage[font={footnotesize}]{caption}
\usepackage{framed}
\usepackage{listings}
\usepackage[titletoc,title]{appendix}

\newcommand{\mc}{\mathcal}
\newcommand{\mb}{\mathbb}

\newcommand{\R}{\mb R}

\newcommand{\N}{\mb N}
\newcommand{\Z}{\mb Z}

\newcommand{\T}{\mb T}

\newcommand{\eea}{\end{align}}

\renewcommand{\epsilon}{\varepsilon}
\renewcommand{\bar}{\overline}
\renewcommand{\tilde}{\widetilde}

\renewcommand{\phi}{\varphi}

\renewcommand\upsilon{\theta}

\usepackage[normalem]{ulem}

\newtheorem{theorem}{Theorem}[section]

\newtheorem{corollary}[theorem]{Corollary}
\newtheorem{lemma}[theorem]{Lemma}
\newtheorem{proposition}[theorem]{Proposition}
\theoremstyle{definition}
\newtheorem{definition}[theorem]{Definition}

\newtheorem{remark}[theorem]{Remark}

\newtheoremstyle{algorithm}
{4pt}
{4pt}
{}
{}
{}
{:}
{\newline}
{}

\usepackage{xcolor}

\newtheorem{algorithm}{Algorithm}

\newcommand{\balgorithm}{\begin{algorithm}\begin{framed}\ }
\newcommand{\ealgorithm}{\end{framed}\end{algorithm}}

\newcommand{\pippo}{\begin{code} \  \begin{lstlisting}[frame=single]}
\newcommand{\pappo}{\end{lstlisting} \end{code}}

\newcommand{\bd}{\begin{definition}}
\newcommand{\ed}{\end{definition}}

\newcommand{\bt}{\begin{theorem}}
\newcommand{\et}{\end{theorem}}
\newcommand{\bp}{\begin{proposition}}
\newcommand{\ep}{\end{proposition}}
\newcommand{\bc}{\begin{corollary}}
\newcommand{\ec}{\end{corollary}} 
\newcommand{\bl}{\begin{lemma}}
\newcommand{\el}{\end{lemma}}
\newcommand{\br}{\begin{remark}}
\newcommand{\tw}{\mathrm}
\newcommand{\er}{\end{remark}}

\DeclareMathOperator{\Id}{Id}

\usepackage{enumitem}
\begin{document}

\title[Coupling of chaotic and gradient systems]{Nonuniformly hyperbolic systems arising from coupling of chaotic and  gradient-like systems}
\author{Matteo Tanzi} 
\address{Courant Institute of Mathematical Sciences, New York University, New York, NY 10012, USA}
\author{Lai-Sang Young}
\address{Courant Institute of Mathematical Sciences, New York University, New York, NY 10012, USA, and Institute for Advanced Study, Princeton, New Jersey 08540, USA}
\thanks{\emph{2010 Mathematics Subject Classification.}  Primary: 37C40; Secondary: 37D25, 37D30.}
\thanks{\emph{Key words and phrases.} Coupled dynamical systems, nonuniformly hyperbolic systems,
SRB measures.\looseness=-9}
\thanks{LSY was supported in part by NSF Grant DMS 1901009}

\maketitle

\begin{abstract}
We investigate  dynamical systems obtained by coupling two maps, one of which is chaotic and is exemplified by an Anosov diffeomorphism, and the other is of gradient type and is exemplified by a N-pole-to-S-pole map of the circle. Leveraging techniques from the geometric and
ergodic theories of hyperbolic systems, we analyze three different ways of coupling together the two
maps above. For weak coupling, we offer an addendum to existing theory showing that almost
always the attractor has fractal-like geometry when it is not normally hyperbolic. Our main results
are for stronger couplings in which the action of the Anosov diffeomorphism on the circle map
has certain monotonicity properties. Under these conditions, we show that the coupled systems
have invariant cones and possess SRB measures even though there 
are genuine obstructions to uniform hyperbolicity.
\end{abstract}

\section{Introduction}
Coupled dynamics occur naturally in
the modeling of many systems of broad interest in contemporary science,  see e.g. \cite{PRK01,SS93, K75}. 
Large and complicated models of real-world systems can often be decomposed into smaller, more
tractable subsystems that interact with one another. Studying these constituent
subsystems and their interactions may offer a way to gain insight into the
larger system. Moreover, a wealth of new examples can
be obtained by coupling together dynamical systems with known properties
and by leveraging knowledge of the subsystems to describe the composite system.

Systems of coupled maps have  been extensively investigated, first numerically, e.g. \cite{K93}, and later theoretically within the framework of ergodic theory; see
\cite{CF05}. Theoretical investigations started with the study of infinite lattices of (weakly interacting) coupled maps, first by Bunimovich and Sinai \cite{BS88},  followed by numerous other authors, see e.g.\ \cite{BD-EIJK98,BK96,FR00,KK92,KL06}. The analysis of coupled maps was later extended to include finite size networks with variable interaction graphs \cite{BLMCH06,KY10, PerTanzHeterCoupl}.
In a majority of the results above, the constituent subsystems are copies of 
a single map (e.g.\ circle rotation, expanding map, piecewise
expanding maps and the like). Moreover, the coupling strength is often -- though not always -- assumed to be sufficiently weak so that the dynamics resemble those of the uncoupled system;
see eg.\ \cite{BKZ09,F20,SB16}.

 While homogeneous materials in physics inspired the study of systems in which 
identical maps are coupled, 
there are many examples from e.g. biology where the components of
a network represent different substances (e.g. enzymes and substrates) or interacting agents (e.g. neurons)  with distinct characteristics and functions, and some of these components
 can influence others in very substantial ways (e.g. \cite{ dJ02, wang2009metabolism, young2020towards}).
These examples suggest the investigation of {\it inhomogeneous networks}, inhomogeneous
in the sense that 
the constituent subsystems may be unequal and the coupling among them not necessarily
symmetric.

This paper is a small step in the direction of inhomogeneous networks.
It analyzes the coupling of two maps that are in some sense at opposite ends of the 
dynamical spectrum.  Individually, both maps are very simple and much studied:
one is an Anosov diffeomorphisms of the two-torus, representing  chaotic systems, and the other is a N-pole-to-S-pole map of the circle, representing orderly, gradient-like dynamics. 

As we will show,  the coupled system can exhibit interesting behavior 
that depends
crucially on the nature of the interaction between the constituents.  
We will consider three types of coupling: (i) weak couplings,
by which we mean small perturbations of the  uncoupled system, (ii) regular couplings
in which the Anosov map acts on the gradient circle map in a regular way
by rotations that can be quite large, and (iii) rare but strong interactions  similar to 
those in (ii) except that the two subsystems do not interact most of the time and the
interaction is more ``dramatic" when it occurs. 
We identify natural conditions on the coupling
under which we prove that the coupled system has nice statistical properties including the
existence of SRB and physical measures, even though the geometric picture can be 
wild.

In a broad sense, the goal of this paper is to  promote the use of 
dynamical systems with known properties as building blocks for larger and more complex systems; the
constituent subsystems need not be copies of the same map, and the coupling can be strong
or weak. This is an excellent way to create many high-dimensional examples that are
novel and potentially within the reaches of analytical techniques.
Closer to the content of the present paper are examples of nonuniformly hyperbolic systems,
often-mentioned candidates of which are the standard map, 
H\'enon maps 
\cite{benedicks1991dynamics, benedicks1993sinai} or more generally rank-one attractors 
\cite{wang2008toward, lu2013strange}. But 
there are also many naturally occurring examples of nonuniformly hyperbolic systems 
that are more amenable to analysis as our study demonstrates.
These examples tend to have certain monotonicity properties which lead to invariant cones,
but may have obstructions to uniform hyperbolicity as is the case with
 couplings (ii) and (iii) in this paper. 

Finally, we mention that parts of our results have overlaps with \cite{B18,H12, Y93} but our constructions are more general and more generalizable; see Sect. \ref{Sec:LiterRev} for the relations of this paper with the the literature on hyperbolic dynamics,  and Sect. \ref{Sec:Generalizations} for a discussion of further results covered by our analysis.

\smallskip
\paragraph{\bf Acknowledgments: } the authors are grateful to Bastien Fernandez for many useful discussions.


\section{Setup  and overview}

In Sect. \ref{Sec:Const} we set some notation and introduce the two maps to be coupled.
In Sect. \ref{Sec:Overv} we discuss the couplings and give an overview of the results.

\subsection{ The two constituents: an Anosov diffeomorphism and a gradient-like map}\label{Sec:Const}

\ \medskip  

In what follows $\T=\R/\Z$ is the unit circle positively parametrized and $\T^2$ is the two-dimensional
torus. The systems  to be coupled are a $C^2$  Anosov diffeomorphism $A:\mb T^2\circlearrowleft$
and a $C^2$ map $g:\mb T\circlearrowleft$  that resembles the time-$t$ map of a 
``north-south flow"  of a gradient-like vector field (see below).  

Anosov diffeomorphisms are characterized by the presence of uniform expansion
and contraction everywhere on their phase spaces. Here we let 
$T_{\tw p}\T^2= E_A^s(\tw p)\oplus E_A^u(\tw p)$ be the invariant splitting 
of the tangent space at $\tw p \in \T^2$, i.e. 
 \[
DA_{\tw p} E_A^s(\tw p)=E_A^s(A\tw p)\quad\mbox{and}\quad DA_\tw p E_A^u(\tw p)=  E_A^u(A\tw p),
 \]
where $DA_{\tw p}$ is the differential of $A$ at $\tw{p}$,    and for simplicity we assume that 
the expansion and contraction occur in one step, i.e., there exists
$\lambda_A < 1$ such that for all $\tw p \in \T^2$,
\begin{eqnarray*}
\|DA_{\tw p}w\|\le \lambda_A\|w\| \quad & \mbox{for all } & w\in E_A^{s}(\tw p) \\
\|D(A^{-1})_{\tw p}v\|\le \lambda_A\|v\| \quad & \mbox{for all } & v\in E_A^{u}(\tw p) .
\end{eqnarray*}
Stable and unstable manifolds are defined everywhere on $\T^2$, and it is a known fact
that all Anosov diffeomorphisms of $\T^2$ are topologically transitive \cite{hasselblatt2002handbook}.

As for the map $g$, we assume it has two fixed points, one repelling, 
called $z_r$, and the other one attractive, called $z_a$. 
We restrict our attention to the case  where $g$ is orientation preserving so that $g'>0$ with respect to the parametrization above, and let
 \[
 \lambda_g^{max}:=\max g'=g'(z_r) \quad\mbox{ and }\quad \lambda_g^{min}:=\min g'=g'(z_a).
 \]  
We assume also that all orbits originate from $z_r$ and are attracted to $z_a$,
i.e., for all $z \in \T \setminus \{z_a, z_r\}$, 
$g^n z \to z_a$, and $g^{-n}z \to z_r$ as $n \to \infty$.  

The uncoupled system is $F_0:\T^2\times \T\circlearrowleft $ given by
\begin{equation}\label{Eq:UncMap}
F_0(x,y,z)=\left(A(x,y), g(z)\right).
\end{equation}
 We will refer to $A$ as the \emph{base map}, 
$g$ as the {\it fiber map}, and for 
$(x,y,z)\in\T^2\times \T$ we call $(x,y) \in \T^2$ the 
\emph{horizontal} component and $z \in \T$ the \emph{vertical} component.

It follows that the uncoupled system $F_0$ has a uniformly hyperbolic attractor at 
$\T^2 \times \{z_a\}$, and a uniformly hyperbolic repellor at $\T^2 \times \{z_r\}$.

\subsection{Couplings  and overview of results}\label{Sec:Overv}

\ \medskip

Our first results are for {\it weak couplings}, which translate
into small perturbations  of the uncoupled map $F_0$. We will focus on the effect of the coupling on 
$\Lambda_0:=\T^2 \times \{z_a\}$, the uniformly hyperbolic attractor of $F_0$. 
 The theory of such attractors is well established including their persistence under small
perturbations. Here we add only the fact that when $\Lambda_0$ is not normally hyperbolic,
most perturbations lead to attractors with fractal-like geometry. This is discussed in 
Sect. \ref{Sec:SmallInte}.

Then we move to stronger couplings,  considering the case where fiber dynamics
are driven strongly by the base map $A$ but feedback to $A$ is weak. That is to say, 
we consider small perturbations of maps $F : \T^2\times \T \circlearrowleft$ of the form
\begin{equation} \label{skewprod}
F(x,y,z) = (A(x,y), g_{(x,y)}(z))
\end{equation}
 where for each $(x,y) \in \T^2$, $g_{(x,y)}: \T \circlearrowleft$ is a diffeomorphism mapping
the fiber at $(x,y)$ to the fiber at $A(x,y)$; the maps $g_{(x,y)}$ vary with $(x,y)$. 
 One-sided couplings such as that in (\ref{skewprod}) in which the dynamics 
of the base drive the dynamics on fibers but not {\it vice versa} are called {\it skew products}.

In Sects. \ref{Sect:RegMonotINtGeom} and \ref{Sec:RegMonotStatist}, we study examples 
with  $g_{(x,y)} = g \circ r$ where $g$ is as defined in Sect. \ref{Sect:DynGeomDesc} and
$r$ is a rotation of the fiber by an amount depending on $(x,y)$.
Provided that $r$ varies monotonically along the unstable directions of $A$ (see Sect. \ref{Sec:StandingAssum} for precise conditions), we show that these maps satisfy a
domination condition,  even though in the ``central" direction the derivative is sometimes
expanding and sometimes contracting. Under suitable conditions, we prove the existence
of an open set of nonuniformly hyperbolic maps with SRB measures.

In Sect \ref{Sec:RareStrongInterac} we also consider monotonic rotations, but here 
$A$ and $g$ do not interact most of the time, in the sense that $r$, the amount rotated, is identically
equal to zero for most $(x,y) \in \T^2$ and it climbs steeply from $0$ to $1$ when the two
maps do interact. 
We call these {\it rare but strong} interactions. These systems are farther away from  uniformly hyperbolic systems than those studied in Sects. \ref{Sect:RegMonotINtGeom} and \ref{Sec:RegMonotStatist}, but we prove that they possess SRB measures nevertheless.

To summarize, we consider in this paper three couplings the first one of which produces
a uniformly hyperbolic attractor and the second and third lead to
dynamical pictures that remain tractable but are successively farther from uniform hyperbolicity.
These results are presented for two constituent maps $A$ and $g$ for definiteness, 
but our proofs in fact apply with  
minor modifications to a number of other situations some of which will be discussed
in Section \ref{Sec:Generalizations}.


\section{Small Interactions}\label{Sec:SmallInte}

\subsection{Dynamical and geometrical description}\label{Sect:DynGeomDesc}

\ \medskip

The case of a small interaction between $A$ and 
$g$ is
described by 
a diffeomorphism $F : \T^2\times \T\circlearrowleft$ that is a small
perturbation of $F_0$  as defined in Eq. \eqref{Eq:UncMap}, and
recall that $\Lambda_0$ is the attractor $\T \times \{z_a\}$ for $F_0$.

\medskip
\begin{theorem}[ mostly \cite{robinson1976structural, ruelle1976measure}]\label{Thm:SmallPert}
Let $F$ be a sufficiently small $C^1$ perturbation of $F_0$. Then:
\begin{itemize}
\item[(i)] $F$ has a uniformly hyperbolic attractor $\Lambda$ near $\Lambda_0$;
\item[(ii)] $F|_{\Lambda}$ is topologically conjugate to $F_0|_{\Lambda_0}$;
\item[(iii)] if $F$ is $C^{1+\epsilon}$ for some $\epsilon>0$, then it admits  a unique SRB measure $\mu$
supported on $\Lambda$, and 
 there is a full Lebesgue measure set $V \subset \T^2 \times \T$ 
such that the following holds: Let $\varphi: \T^2 \times \T \to \R$ be
any continuous function. Then
\[
\lim_{n\rightarrow\infty}\frac1n \sum_{i=0}^{n-1} \varphi(F^ip)=\int_\Lambda \varphi d\mu \quad\quad\forall p\in V.
\]
\end{itemize}
\end{theorem}

\begin{proof}
(i) and (ii) are classical results \cite{robinson1976structural} (see also \cite{hasselblatt2002handbook}). For (iii), 
The existence of an SRB measure supported on $\Lambda$ was proved in \cite{ruelle1976measure}.
By the absolute continuity of the stable foliation  \cite{hirsch1977invariant}, there is a set 
$V$ with the properties above occupying a full Lebesgue measure subset of $\mathcal B(\Lambda)$, the basin of attraction of $\Lambda$. 
To complete the proof, it remains  to show that $(\T^2 \times \T) \setminus \mathcal B(\Lambda)$ 
has Lebesgue measure zero. By an analogous argument applied to $F^{-1}$, we know that 
$(\T^2 \times \T) \setminus \mathcal B(\Lambda)$ is a uniformly hyperbolic repellor, and such
repellors are known to have 
Lebesgue measure zero; see  e.g.  \cite{young1990large}.
\end{proof}

\smallskip

While the dynamics are essentially unchanged in the sense that $F|_\Lambda$ is topologically
conjugate to $F_0|_{\Lambda_0}$, the geometry of the attractor $\Lambda$ 
can be quite different than that of $\Lambda_0$. 
As we will show, what determines the geometry of $\Lambda$ is
the relative strengths of
contraction in the Anosov map $A$ and at the sink of the fiber map $g$. We distinguish 
 between the following two cases:

\medskip \noindent
1. {\it  Normally hyperbolic attractor.} Assume
\[
\min \{\|DA_{\tw p} v\|: \tw p \in \T^2, v \in E_A^s(\tw p), \|v\|=1\} > \lambda_g^{min}\ .
\]
Then $\Lambda_0$ is normally hyperbolic under $F_0$, and it is a well known fact
that normally hyperbolic manifolds persist under small perturbations \cite{hirsch1977invariant}. That is,
for $F$ sufficiently close to $F_0$ in the $C^1$-sense, $\Lambda$ is again diffeomorphic to
a smooth 2D torus.

\bigskip \noindent
2. {\it Attractors with wild geometry.} Assume
\begin{equation} \label{domination}
\alpha_{max}^s:= \max \{\|DA_{\tw p} v\|: \tw p \in \T^2, v \in E_A^s(\tw p), \|v\|=1\} < \lambda_g^{min}\ .
\end{equation}

 When the perturbed map $F$ has a skew-product structure, it has been shown
 that the attractor $\Lambda$ can be the graph of a nowhere 
differentiable function with fractal dimension; see  \cite{kaplan1984lyapunov}, and later  
\cite{hadjiloucas2002regularity}. Below we prove a result along similar lines 
for small perturbations that allow feedback, i.e., the perturbed map $F$ is not
necessarily a skew-product.

\medskip
 The rest of this section is about case 2. We thus work 
under the assumption that $\alpha_{max}^s < \lambda_g^{min}$. 
For the uncoupled map $F_0$, there is a splitting  
of the tangent space at $p \in \Lambda_0$ into 
\begin{equation} \label{splitting}
T_p (\T^2 \times \T) = E^s(p) \oplus E^c(p) \oplus E^u(p)
\end{equation}
where $E^s = E^s_A \times \{0\}$, $E^u = E^u_A \times \{0\}$, and $E^c$ is
in the vertical direction. We have called $E^c$ the ``central direction" 
but $DF|_{E^c}$ is in fact strictly contracting though not as strongly as 
$DF|_{E^s}$. 
It follows from the stability of uniformly hyperbolic splittings that  for $F$
sufficiently near $F_0$ in the $C^1$-sense, the splitting in (\ref{splitting}) persists on $\Lambda$, 
with
\begin{eqnarray*}
\max_{p \in \Lambda} \|DF_p|_{E^s(p)}\|
& < & \min_{p \in \Lambda} \|DF_p|_{E^c(p)}\|\\
\mbox{and} \qquad \max_{p \in \Lambda} \|DF_p|_{E^c(p)}\| & < & 1 \ < \ 
\min_{p \in \Lambda} \|DF_p|_{E^u(p)}\|\ .
\end{eqnarray*}

The following notation will be used.
For $p_* \in \T^2 \times \T$, we let $B_\epsilon(p_*)$ denote
the $\epsilon$-neighborhood of $p_*$ in $\T^2 \times \T$. In local analysis
we identify $\T^2 \times \T$ with $\mathbb R^3$, and for
$p_1, p_2 \in \T^2 \times \T$, we let $\overrightarrow{p_1p_2}$ be the 
vector from $p_1$ to $p_2$. Let Diff$^1(\T^2 \times \T)$ be
the space of $C^1$ diffeomorphisms of $\T^2 \times \T$ onto itself,
endowed with the $C^1$
topology.

\medskip
\begin{theorem}\label{Thm:SmallPertGeom}
Assume $\alpha_{max}^s < \lambda_g^{min}$. Then there is 
a neighborhood $\mathcal N$ of $F_0$ in Diff$^1(\T^2 \times \T)$ and 
an open and dense subset $\mc N' \subset \mc N$, 
such that for all $F \in \mathcal N'$, the attractor has
the following property:

\medskip \noindent
(*) \ \ for every $p_*\in\Lambda$ and every $\epsilon,\delta>0$, $\exists p_1, p_2\in B_\epsilon(p_*) \cap \Lambda$ s.t. $\measuredangle\left(\overrightarrow{p_1p_2}, E^c(p_*)\right)<\delta$.
\end{theorem}

\medskip
Our rationale for property (*) is as follows:
In the case where $F$ is a skew-product and $\Lambda=$ graph$(h)$ for 
some $h: \T^2 \to \T$, one way to see that $h$ is nowhere differentiable is
the presence of  ``arbitrarily large slopes" everywhere in the domain.
As $h$ is a nondifferentiable function,
``arbitrarily large slopes"  here refers to linear segments connecting points in 
the graph of $h$ that are arbitrarily close to being vertical.
This is essentially what (*) is intended to capture, except that for the perturbed map,
$E^c$ in general no longer coincides with the ``vertical" direction. Replacing the vertical fibers by 
the central foliation $W^c$, we observe that even though $\Lambda$ intersects every central leaf $W^c$ in one point so $\Lambda$
can in principle be viewed as the graph of a function, we have not pursued the
regularity of this graphing function because the central foliation is 
in general only H\"older continuous \cite{hirsch1977invariant}.


\subsection{Proof of Theorem \ref{Thm:SmallPertGeom}}

\ \medskip

The domination condition $\alpha_{max}^s < \lambda_g^{min}$ is assumed throughout.
We begin by recalling some standard facts on invariant manifolds from uniform
hyperbolic theory; see \cite{hirsch1977invariant}. The map $F$ has the following dynamical structures:
At each $p \in \Lambda$, there  is 

\smallskip
-- a stable manifold $W^s(p)$, 

-- a center stable manifold $W^{cs}(p)$, 

-- a center manifold $W^c(p)$ and 

-- an unstable manifold $W^u(p)$ 

\smallskip
\noindent tangent to $E^s(p), \ E^s(p) \oplus E^c(p), \ E^c(p)$ and $E^u(p)$ 
respectively. As noted in Sect. \ref{Sect:DynGeomDesc}, the dynamics on $W^c$ and $W^{cs}$ are
in fact strictly contractive, though not as strongly as they are on $W^s$.
 $\Lambda$ is a uniformly hyperbolic attractor and for any $p\in\Lambda$, $W^u(p)\subset \Lambda$.
The subspaces $E^s$ and $E^{cs}$ are in fact defined everywhere on  $\mathcal B(\Lambda)$, 
the basin of attraction of $\Lambda$, as they depend solely on a trajectory's forward iterates.
On  $\mathcal B(\Lambda)$, the subspaces $E^s$ and $E^{cs}$ vary continuously, and  
there are two continuous families of invariant manifolds
$W^s$ and $W^{cs}$ tangent everywhere to $E^s$ and $E^{cs}$,
with $W^s$-leaves foliating smoothly each $W^{cs}$-manifold; 
see \cite{hirsch1977invariant}. To avoid confusion, we will sometimes we will sometimes make explicit the dependence of the manifolds on the map $F$ and write  $W^*(p;F)$, $*\in\{s,c,cs,u\}$.  

As we will see,  the geometry of $\Lambda$ has much to do with whether 
or not $W^s(p) \subset \Lambda$ for $p \in \Lambda$,
and for that, we have the following dichotomy:

\begin{proposition}\label{Prop:strongstable}
Under the assumptions of Theorem \ref{Thm:SmallPertGeom},  if 
 $W^{s}(p)\subset \Lambda$ for some $p\in\Lambda$,
 then  $W^{s}(q)\subset \Lambda$ for all $q\in \Lambda$. 
\end{proposition}

\begin{proof} Let $p$ be as above, and let $\Phi:\Lambda_0 \to \Lambda$ be the
topological conjugacy between $F_0|_{\Lambda_0}$ and $F|_\Lambda$, with 
$\Phi(p_0)=p$. Let $ W_0^s( p_0)$ and $ W_0^{cs}( p_0)$ denote $W^s(p_0;F_0)$ and $W^{cs}(p_0;F_0)$ repectively. Then $\Phi( W_0^{cs}( p_0))=W^{cs}(p)$ 
since $ W_0^{cs}( p_0)$ and $W^{cs}(p)$ are the points that asymptotically converge 
to $ p_0$ and $p$ respectively. 
 Since $ W_0^s( p_0)=  W_0^{cs}( p_0)\cap \Lambda_0$, it follows that
 $\Phi( W_0^s( p_0)) = W^{cs}(p) \cap \Lambda$, and this must be equal to $W^s(p)$
since each central leaf $W^c$ intersects $\Lambda$ at exactly one point and 
$W^s(p) \subset \Lambda$ by assumption. Thus for the point $p$, we have 
$\Phi( W_0^s( p_0)) = W^s(p)$. 

For $q \in \Lambda$ and $r>0$, let $W^s_r(q)$ denote the local stable manifold 
of radius $r$ centered at $q$. To prove the proposition it suffices to show that 
for some $r>0$,  $W^s_r(q) \subset \Lambda$ for all $q \in \Lambda$. 
Let $q \in \Lambda$ and $r>0$ be fixed. As $ W_0^s( p_0)$
is dense in $ \Lambda_0$,  $W^s(p)$ in dense in $\Lambda$. We pick
$q_n \in W^s(p)$ such that $q_n \to q$. By  uniform hyperbolicity of $\Lambda$ and  
continuity of  local stable disks, $W^s_r(q_n)$ tends to $W^s_r(q)$ 
as $n \to \infty$. Since  $W^s_r(q_n) \subset W^s(p) 
\subset \Lambda$ and $\Lambda$ is a closed set, it follows that $W^s_r(q) \subset
\Lambda$.
\end{proof}

\smallskip

\begin{proof}[Proof of Theorem \ref{Thm:SmallPertGeom}]
The set $\mathcal N$ consists of $F \in $Diff$^1(\T^2 \times \T)$ sufficiently close
to $F_0$ so that the domination conditions (discussed in Sect. \ref{Sect:DynGeomDesc}) are preserved
on an open set containing the attractor $\Lambda$.
Shrinking $\mc N$ if necessary, we will show that  $\mc N'$ is the set consisting of
those $F \in \mc N$ with the property that $W^{s}(p) \not \subset \Lambda$ for all $p$.

\medskip
First we assume $W^s(p) \not \subset \Lambda$ for all $p \in \Lambda$, and prove
property (*). The following hyperbolic estimate is standard. 

\begin{lemma} The following holds for all $F$ $C^1$-sufficiently close to $F_0$.
Let $\varepsilon_1>0$ be small enough that
$\alpha_{max}^s+\varepsilon_1 < \lambda_g^{min}-\varepsilon_1$, and let
$r=r(\varepsilon_1)>0$ be small enough. Then there exists $C \ge 1$
such that for $p \in \Lambda$ and $q \in W^{cs}_{r/2}(p)$, there is a unique point
$q^c \in W^c_r(p) \cap W^s_r(q)$. Moreover,
\begin{eqnarray*}
d(F^nq, F^nq^c) & \le & C (\alpha_{max}^s+\varepsilon_1)^n d(q, q^c), \\
\mbox{and} \qquad d(F^nq^c, F^np) &  \ge & C^{-1} (\lambda_g^{min}-\varepsilon_1)^n d(q^c,p)\ .
\qquad \mbox{for all } n \in \mathbb Z^+\ .
\end{eqnarray*}
\end{lemma}

The relevance of this estimate is as follows: Assuming $q^c \ne p$, we have, 
 as $n \to \infty$, $d(F^nq, F^nq^c) \ll d(F^nq^c, F^np)$, so
 $\measuredangle\left(\overrightarrow{(F^np)(F^nq)}, \overrightarrow{(F^np)(F^nq^c)}\right)\ll 1$.
At the same time, the length of $F^n(W^c_r(p))$ decreases exponentially with $n$, 
so $\overrightarrow{(F^np)(F^nq^c)}$ becomes increasingly aligned with
$E^c(F^np)$. This implies
$\measuredangle\left(\overrightarrow{(F^np)(F^nq)}, E^c(F^np)\right) \to 0$.

To prove (*), we pick $p_0 \in \Lambda$ with the property that its forward orbit is dense in $\Lambda$.  Recall that for any $r'>0$, $W_{r'}^{cs}(p_0)\cap \Lambda\neq \emptyset$
 and is homeomorphic to a segment via the topological conjugacy with the unperturbed attractor.
If $W_{r'}^{cs}(p_0)\cap \Lambda\subset W^s(p_0)$ then $W_{r}^s(p_0)\subset \Lambda$  for some $r>0$, and it would follow that $W^s(p_0) \subset \Lambda$, which we have ruled out.
So for any $r>0$,  there exists $q_0 \in \Lambda \cap W^{cs}_{r/2}(p_0)$ 
such that $q_0 \not \in W^s(p_0)$. This implies $q_0^c \ne p_0$, and the segments 
$\overrightarrow{(F^np_0)(F^nq_0)}$ will be increasingly aligned with $E^c$ as $n \to \infty$.
The orbit of $p_0$ carries these segments to every $B_\varepsilon(p_*)$, proving (*).

\begin{figure}[h!]
\center
\includegraphics[scale=0.5, page=7]{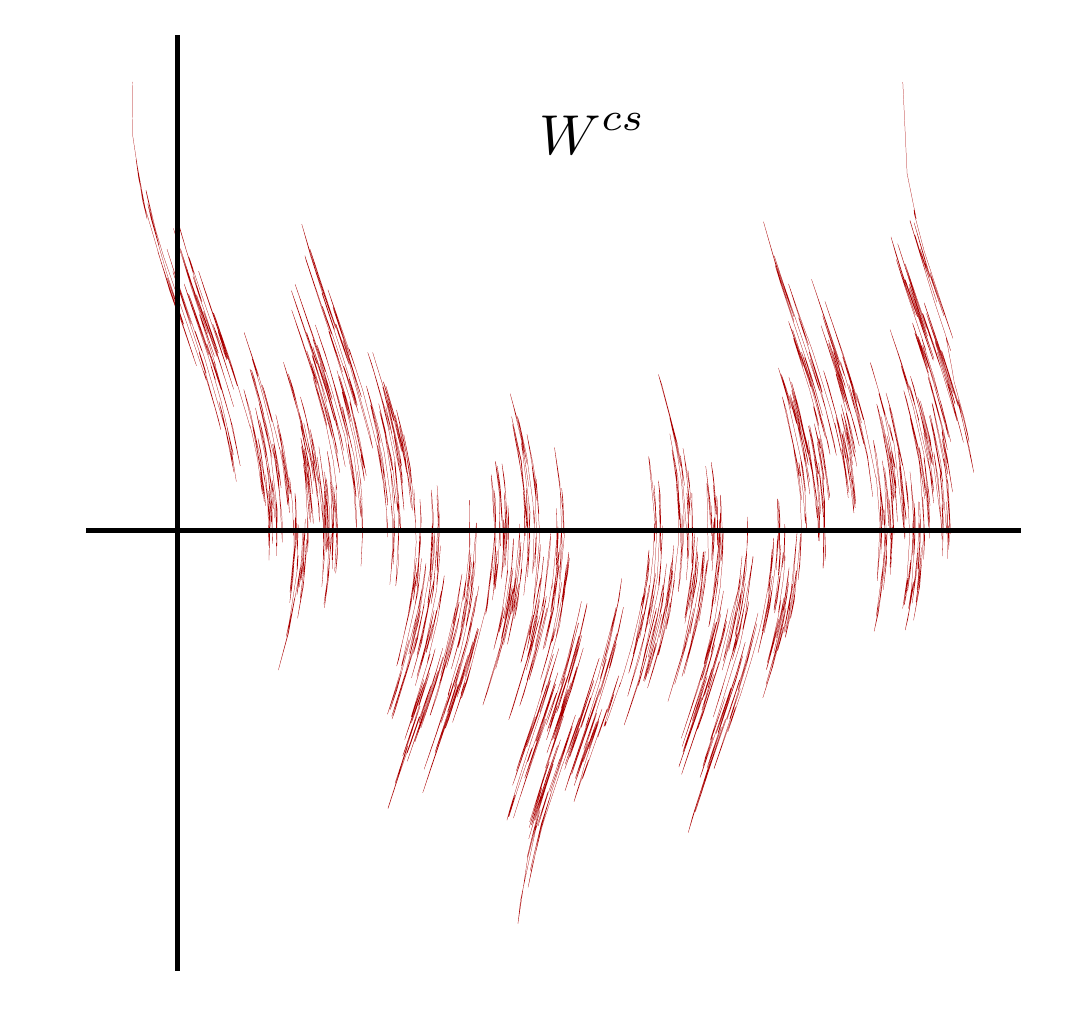}
\caption{Illustration of a coupled system having an attractor with
wild geometry. The plotted points represent the intersection of the attractor $\Lambda$ with a $W^{cs}-$manifold through one of the points in the attractor. The horizontal and vertical axes indicate approximate directions of $E^s$ and $E^c$.}
\end{figure}

\medskip
It remains to prove that $\mc N'$ is open and dense in $\mc N$.

To prove openness, assume that for a given $F$, $W^s(p') \not \subset  \Lambda$ for all $p' \in \Lambda$. Let $\bar p \in \Lambda$ be a fixed point. Then  there is $q\in W^s_r(\bar p)\backslash \Lambda $ with  $d(q,\Lambda)>0$. The fixed point $\bar p$, $W^s_r(\bar p)$, and $\Lambda$ all vary continuously as we perturb the map $F$. Therefore one can always find a point $q$ with the same property as above for any small enough perturbation. 
 By Proposition \ref{Prop:strongstable}, this implies that if $G$ is a small enough perturbation of $F$
and $\Lambda_G$ is the attractor of $G$, then
$W^s(p') \not \subset  \Lambda_G$ for all $p' \in \Lambda_G$.

 To prove denseness, suppose that $F \in \mc N \setminus \mc N'$; we will produce 
 an arbitrarily small perturbation of $F$ that takes it into $\mc N'$  by
precomposing with the time-$t$-map of a flow to be constructed.
 Let $\bar p \in \Lambda$ be a fixed point, 
 and pick a periodic point $q \in \Lambda$ so that 
 $W_r^{s}(\bar p)$ meets $W_r^u(q)$ at a unique point $p_*$  Now $W^u_r(q)\subset \Lambda$
  since unstable manifolds of uniformly hyperbolic attractors are always contained in the attractor.
 This and the assumption that $W^s_r(p)\subset \Lambda$ imply  $p_*\in\Lambda$. 
 Let $B$ be a small ball centered at $F(p_*)$. We assume it is small enough that
it does not meet either $F^{-i}W_r^u(q)$ for $0\leq i<\tau$ where $\tau$
is the period of $q$, or the part of
$W^s_r(\bar p)$ containing $F^i(p_*)$ for $i >1$.  Choose a vector field so that is identically
zero outside of $B$; in $B$, it points upwards in the vertical component and is
tangent to $W^{cs}_r(\bar p)$. Let $\varphi_t$ be the generated flow , 
and let $G= \varphi_t \circ F$ for an arbitrarily small $t>0$. This ensures that $\bar p, q$ 
are untouched, so they are in $\Lambda_G$, the attractor of $G$. We also have $W^u_r(q;G)=W^u_r(q)$ and 
$W^{cs}_r(\bar p;G)=W^{cs}_r(\bar p)$. So $W^{cs}_r(\bar p; G) \cap W_r^u(q; G)=\{p_*\}$
as before, but $W^s_r(\bar p;G)$ now meets $W^c_{\rm loc}(p_*; G)$ at a point $p_*'$ 
a little below $p_*$.
For $G$ to remain in $\mc N \setminus \mc N'$, $\Lambda_G$ must contain $W^s(\bar p;G)$,
hence $p_*' \in \Lambda_G$. But then $W^c_{\rm loc}(p_*;G)$ meets $\Lambda_G$ 
in at least two points, which cannot happen, so $G \in \mc N'$ after all.
\end{proof}


\section{ Regular Monotonic Interactions: Geometrical Properties}\label{Sect:RegMonotINtGeom}

In  Sections \ref{Sect:RegMonotINtGeom} and \ref{Sec:RegMonotStatist} we consider small perturbations of skew-product
diffeomorphisms $F$ having the form
\begin{equation} \label{Eq:SkewF}
F(x,y,z)=\left(A(x,y), g(z+r(x))\right)
\end{equation}
where $r:\T\rightarrow \T$ is a orientation preserving diffeomorphism,  so $r'>0$. This equation
represents a coupling in which
the Anosov map in the base drives the fiber maps by rotating the circle at $(x,y)$ by
an amount equal to $r(x)$.
In Sect. \ref{Sec:RegMonotStatist}  we show that couplings of this kind produce open
sets of maps with SRB measures on attractors that are not uniformly hyperbolic.
 The present section lays the geometric groundwork.
 The results of these two sections depend crucially
on the {\it monotonicity} in the amount rotated, 
i.e., on the monotonic dependence of $r$ on $x$.

Notationally though not necessarily in substance, it is cleaner to first present 
 some of
this material for the skew-product map $F$ above, and that is what we will do.

\subsection{Assumptions and  Invariant Cones}\label{Sec:AssumDom}

\ \medskip

We begin with some standing
assumptions, introducing some  notation along the way. 
Let $\partial_x, \partial_y, \partial_z$ 
denote the unit vectors pointing in the $x, y$ and $z$-directions respectively.

\subsubsection{ Standing Assumptions for sections  \ref{Sect:RegMonotINtGeom} and \ref{Sec:RegMonotStatist}}  \label{Sec:StandingAssum}
 Let $F$ be the skew-product map defined in Equation (\ref{Eq:SkewF}), and
let $\lambda_A$, $\lambda_g^{min}$ and $\lambda_g^{max}$ be as defined in Section \ref{Sec:Const}.
We assume:

\medskip \noindent
(A1) $\lambda_A<\lambda_g^{min}<1<\lambda_g^{max}<(\lambda_A)^{-1}$ 

\bigskip \noindent
(A2) The map $r$ is $C^2$ with $c \le r' \le C$ for some $c, C>0$.

\bigskip \noindent
(A3) The following two conditions on the geometry of $A$ are assumed:

\medskip \noindent
(i) To ensure that the function $r$ varies monotonically along stable and
(more importantly) unstable manifolds of $A$, we assume at every $\tw p \in \T^2$ that 
$E^u_A(\tw p)$ and $E^s_A(\tw p)$ are not aligned with 
the $y$-axis, and let $\partial_A^u(\tw p)$ be the unit vector in $E^u_A(\tw p)$ with a positive
$x$-component, $\partial_A^s(\tw p)$ the unit vector in $E^s_A(\tw p)$ with a
negative $x$-component. It follows by compactness that there exists $\beta>0$ such that 
\[
\partial_A^u(\tw p) \cdot \partial_x \ge \beta, \quad
\partial_A^s(\tw p) \cdot \partial_x \le - \beta \qquad \mbox{for all } \tw p\in\T^2\ .
\]

\noindent
(ii) Let $\alpha^u, \alpha^s$ be functions on $\T^2$ such that at each $p$,
\[
DA_{\tw p}(\partial_A^u(\tw p))= \alpha^u(\tw p) \partial_A^u(A\tw p), \quad
DA_{\tw p}(\partial_A^s(\tw p))= \alpha^s(\tw p) \partial_A^s(A\tw p) \ .
\]
Assume without loss of generality that $\alpha^u, \alpha^s > 0$, and let $\alpha^{u/s}_{\max}$
and $\alpha^{u/s}_{\min}$ be the maximum and minimum of these functions.

\subsubsection{ Partially hyperbolic splitting}

The fact that $F$ is a skew product implies immediately that $\langle\partial_z\rangle$, the subspace generated by the unit vector in
the $z$-direction, is $DF$-invariant. It implies also that the subspaces 
$E^u_A \oplus \langle\partial_z\rangle$
and $E^s_A \oplus \langle\partial_z\rangle$ are $DF$-invariant.
We write $E^c := \langle\partial_z\rangle$, $E^{cu}: = E^u_A \oplus \langle\partial_z\rangle$ and $E^{cs} := E^s_A \oplus \langle\partial_z\rangle$.

Our first step is to locate at each point a $DF$-invariant subspace $E^u \subset E^{cu}$. 
At each $p \in \T^2 \times \T$, we consider 
\[
DF_p|_{E^{cu}(p)}: E^{cu}(p)
\to E^{cu}(Fp)\ .
\]
For $v \in E^{cu}(p)$, let $v=(v^u, v^z)$ be the components of $v$ 
with respect to the basis $\{\partial_u(p), \partial_z\}$, where $\partial_u(x,y,z)
= (\partial_A^u (x,y), 0)$. For $0 \le s < t \le \infty$, we define the one-sided cone
\[
\mc C^u_{s,t}(p) := \{v \in E^{cu}(p) :
v^u, v^z \ge 0 \mbox{ and } s v^u \le v^z \le t v^u\}.
\]

\begin{lemma}\label{Lem:InvCones} Assume conditions (A1)-(A3), and let 
\[
a = \frac{2\lambda_g^{max} C}{\alpha^u_{min} - \lambda_g^{max}}, \qquad
b = \frac{\lambda_g^{min} c \beta}{\alpha^u_{max}}\ .
\]
Then there exists $c_0>0$ such that the following hold at every $p \in \T^2 \times \T$:
\begin{itemize}
\item[(i)] $DF_p(\mc C^u_{0,a}(p)) \subset \mc C^u_{b,a-\sigma}(Fp)$
for some small $\sigma>0$;
\item[(ii)] for all $v \in \mc C^u_{0,a}(p)$, $\|DF^n_p v\| \ge c_0 (\lambda_A)^{-n} \|v\|$ for all
$n \ge 0$. 
\end{itemize}
\end{lemma}

\medskip
Notice that $a>0$, since $\alpha^u_{min} \ge (\lambda_A)^{-1} > \lambda_g^{max}$
by Assumption (A1), and that $0<b<a$. The fact that the invariant cone is in the positive quadrant of $E^{cu}$ is a consequence of the monotonicity as can be seen from the proof below.
\begin{proof} (i) Let $h(x,y,z):=g(z+r(x))$.
With respect to the bases $\{\partial_u, \partial_z\}$, we have, at every point in
$\T^2 \times \T$,
\[
DF|_{E^{cu}}=\left(\begin{array}{cc}
\alpha^u & 0\\
\partial_uh& \partial_zh
\end{array}\right) \ .
\]
Since all the entries are nonnegative, we see immediately that the first quadrant is
preserved. For $v$ with $v^u \ne 0$ and $v^z = av^u$, let
$DF(v)=(\bar v^u, \bar v^z)$. Then
\[
\frac{\bar v^z}{\bar v^u} = \frac{\partial_uh v^u + \partial_z hv^z}{\alpha^u v^u}
\le \frac{(\lambda_g^{max} C) v^u + \lambda_g^{max} (a v^u)}{\alpha^u_{min} v^u}\ .
\]
Substituting in the value of $a$ above, and writing $\lambda =\lambda_g^{max}, \ 
 \alpha=\alpha^u_{min}$, we obtain
\[
\frac{\bar v^z}{\bar v^u} \le \frac{\lambda C + 
\lambda \left(\frac{2\lambda C}{\alpha-\lambda}\right)}{\alpha} 
= \frac{\lambda C\alpha - \lambda^2 C + 2 \lambda^2 C}{\alpha(\alpha-\lambda)}
= \frac{\lambda C}{\alpha-\lambda} \left(1 + \frac{\lambda}{\alpha} \right) < a \ .
\]
Likewise, letting $v^u >0$ and $v^z=0$, we obtain
\[
\frac{\bar v^z}{\bar v^u} \ge \frac{(\lambda_g^{min} c)\beta v^u}{\alpha^u_{max} v^u} = b\ .
\]
This completes the proof of (i).

\smallskip \noindent
(ii) Using the fact that $F$ is a skew-product, we have, for any $v$,
\[
\|DF^n v\| \ge |(DF^n v) \cdot \partial_u| = \|A^nv^u\| \ge (\lambda_A)^{-n} |v \cdot \partial_u|\ .
\]
The assertion follows since $\|v\cdot \partial_u\| \ge \mbox{const} \|v\|$ for $v \in \mc C^u_{0,a}$.
\end{proof} 

Lemma \ref{Lem:InvCones} implies the existence of a {\it bona fide} $DF$-invariant 
unstable subspace $E^u$ defined
everywhere on $\T^2 \times \T$ with uniform expansion.
A similar analysis can be carried out for 
$DF^{-1}_p|_{E^{cs}(p)}: E^{cs}(p) \to E^{cs}(F^{-1}p)$.
We summarize the results as follows:

\smallskip
\begin{proposition} Under Assumptions (A1)-(A3), there is a $DF$-invariant continuous
splitting of the tangent bundle of $\T^2 \times \T$ into $E^u \oplus E^c \oplus E^s$  
with the properties below.
\begin{itemize}
\item[(i)] $E^c = \langle \partial_z \rangle$; for $v \in E^c$,
$\lambda_g^{min} \|v\| \le \|DFv\| \le \lambda_g^{max} \|v\|$ \ .
\item[(ii)] for all $v \in E^u$, we have $\|DF^n v\| \ge c_0 (\lambda_A)^{-n} \|v\|$ for all $n \ge 0$;

\noindent
for all $v \in E^s$, we have $\|DF^n v\| \le c_0^{-1} (\lambda_A)^n \|v\|$ for all $n \ge 0$.
\end{itemize}
\end{proposition}

\smallskip
By Assumption (A1), $E^c$ is a genuine central subspace, satisfying the following
uniform domination condition:
$$ 
\max \|DF|_{E^s}\| < \min \|DF|_{E^c}\|\ ; \quad \max \|DF|_{E^c}\| < \min \|DF|_{E^u}\|\ .
$$
By standard theory, $W^s, W^{cs}, W^c, W^{cu}$ and $W^u$-manifolds tangent to
$E^s, E^{cs}, E^c, E^{cu}$ and $E^u$ respectively are 
defined everywhere and they form invariant foliations on $\T^2 \times \T$.

\subsection{ A few basic properties}

\ \medskip 

This subsection discusses  some properties of the skew-product map $F$ 
satisfying Assumptions (A1)-(A3) and its perturbations.
Let $\pi_{xy}$ and $\pi_z$ be projections of $\T^2 \times \T$ onto the horizontal and vertical fibers 
respectively.

\subsubsection{Geometry of $W^u$-leaves}\label{Sec:GeomWuleaves}

The following  is a direct consequence of the monotonicity of the interaction,
more specifically of the monotonicity of the function $r$ along the unstable manifolds of $A$. 

By construction,  tangent vectors of $W^u$-curves of $F$ lie in the cones $\mc C^u_{b,a}$ (as defined in Lemma \ref{Lem:InvCones}). More precisely, let $J$ be an interval. 
If $\gamma:J\rightarrow \T^2 \times \T$ is a differentiable embedding  
such that $\gamma(J)$ is a piece of $W^u$-curve, then $\gamma'(t)\in \mc C_{a,b}^u(\gamma(t))$ for all $t\in J$. Geometrically, this means that $\pi_{xy} \circ \gamma(J)$ is a subset of an unstable leaf for $A$, while $\pi_z \circ \gamma_n$ winds around the fiber $\T$
monotonically (see Figure~\ref{Fig:AdmissibleLeaves}). Moreover, $(\pi_z \circ \gamma)'/ (\pi_{xy} \circ \gamma)' \in [b,a]$.

More generally, if $J$ is an interval and  $\gamma_0: J \to \T^2 \times \T$  is a $C^1$ curve such that $\gamma_0'(t)\in\mc C_{0,a}^u(\gamma_0(t))$, then defining $\gamma_n:= F^n \circ \gamma_0$,  
$\gamma_n'(t) \in \mc C^u_{b,a}(\gamma_n(t))$ for all $t \in J$, and $(\pi_z \circ \gamma_n)'/ (\pi_{xy} \circ \gamma_n)' \in [b,a]$.

\subsubsection{Mixed behavior of $DF|_{E^c}$}\label{Sec:MixedBehav}

\medskip
A convenient way to understand $DF|_{E^c}$ is to view $F$ as the composite map
\[F = f_3 \circ f_2 \circ f_1\]
where
\begin{eqnarray*} 
f_1(x,y,z) & :=  &(x,y,z+r(x)); \\
 f_2(x,y,z) &:=  &(x,y, g(z)), \\
 f_3(x,y,z) &: = & (A(x,y), z)\ .
\end{eqnarray*}
See Figure~\ref{Fig:AdmissibleLeaves} for an illustration. From this decomposition, one sees immediately that for $p=(x,y,z)$, the restriction of the differential  of $F$
at $p$ to the invariant direction $E^c$ is
\begin{equation} \label{DFc}
DF_{p}|_{E^c} = g'\circ \pi_z\circ f_1(p)\ .
\end{equation}
We introduce here also the idea of a fundamental domain.
Let $V$ be a $W^u$-curve segment, i.e. a bounded piece of a $W^u$-curve. $V$
 is called a {\it fundamental domain} if $\pi_z(V)=\T$
and no proper subset of $V$ has this property.
Recall that $\alpha^u$ is the coefficient of expansion in the direction of $E^u$. 

The following lemma shows that even  as $F$ has a splitting into $E^u \oplus
E^c \oplus E^s$ with a uniform domination condition (see Sect. 4.1.2),
the central direction has inherently mixed behavior  provided the expansion
is strong enough.

\begin{proposition}\label{Prop:DensityLyap} Assume (A1)-(A3). If additionally $\alpha^u_{min}$
is large enough, then for every $W^u$-curve segment $W$,
\begin{itemize}
\item[(i)]  the set $\{p\in W : \liminf_{n \to \infty} \frac{1}{n}\log \left\|DF^n_p|_{E^c}\right\|>0\}$ is dense in $W$;
\item[(ii)] the set $\{p\in W : \limsup_{n \to \infty} \frac{1}{n}\log \left\|DF^n_p|_{E^c}\right\|<0\}$ is dense in $W$.
\end{itemize}
\end{proposition} 

\smallskip
\begin{proof} To prove (i), we will show that given $W$, there exists $p \in W, 
n_0 \in \Z^+$, and $\delta>0$ such that $DF_{F^np}|_{E^c} > 1+\delta$ for all $n \ge n_0$.
This implies denseness of the set in question 
because the argument can be applied to any subsegment $W' \subset W$
of any length.

Fix $\delta\in (0,\lambda^g_{max}-1)$.  Let $I^+:=\{z \in \T: g'(z) > 1+ \delta\}$. 
Since $F^n(W)$ grows in length,  and it winds around the vertical fiber with
a positive minimum speed (Sect. \ref{Sec:GeomWuleaves}), 
there exists $n_0$ such that $F^{n_0}(W)$ contains a fundamental
domain $V$. 
Let $V_{n_0} \subset V$ be such that $\pi_z \circ f_1(V_{n_0}) = I^+$ where $f_1$
is as defined above. Then for all $q \in V_{n_0}$, 
$DF_q|_{E^c} > 1+\delta$. Now if $F(V_{n_0})$ contains a fundamental domain, then 
the argument can be repeated to produce $V_{n_0+1} \subset F(V_{n_0})$ with the property 
that $\pi_z \circ f_1(V_{n_0+1}) = I^+$. The process can be continued indefinitely
provided that at each stage, $F(V_n)$ contains a fundamental domain. Assuming that,
the point $p \in \cap_{n \ge n_0} F^{-n}V_n$ will have the desired property.

To ensure that the procedure in the last paragraph can be continued, recall that if $\gamma$ is a parametrization of a $W^u$-curve, then $(\pi_z \circ \gamma)'/ (\pi_{xy} \circ \gamma)' \in [b,a]$. This implies that 
(a) there exists $\ell_1$ such that any $W^u$-curve $\bar W$
for which $\pi_{xy} \bar W$ has length $\ge \ell_1$ must contain a fundamental domain, and
(b) there exists $\ell_2$ such that if a $W^u$-segment $\bar W$ is such that
$\pi(f_1(\bar W))=I^+$, then $\pi_{xy} \bar W$ must have length $\ge \ell_2$.
It suffices to require $\ell_2 \cdot \alpha^u_{\min} \ge \ell_1$.

To prove (ii), one substitutes $I^+$ with $I^-:=\{z \in \T: g'(z) < 1- \delta\}$, and the 
same proof carries over \emph{mutatis mutandis}. 
\end{proof}

An immediate corollary of Proposition \ref{Prop:DensityLyap} is that $F$  satisfying the hypotheses of
this proposition cannot admit an Axiom A attractor, in fact, the mixed behaviour in the invariant subbundle
$E^c$ makes it impossible
for it to have a uniformly hyperbolic set that contains entire unstable manifolds.

\begin{figure}[h!]\label{Fig:MonotCurve}
\center
\includegraphics[page=2,scale=0.35]{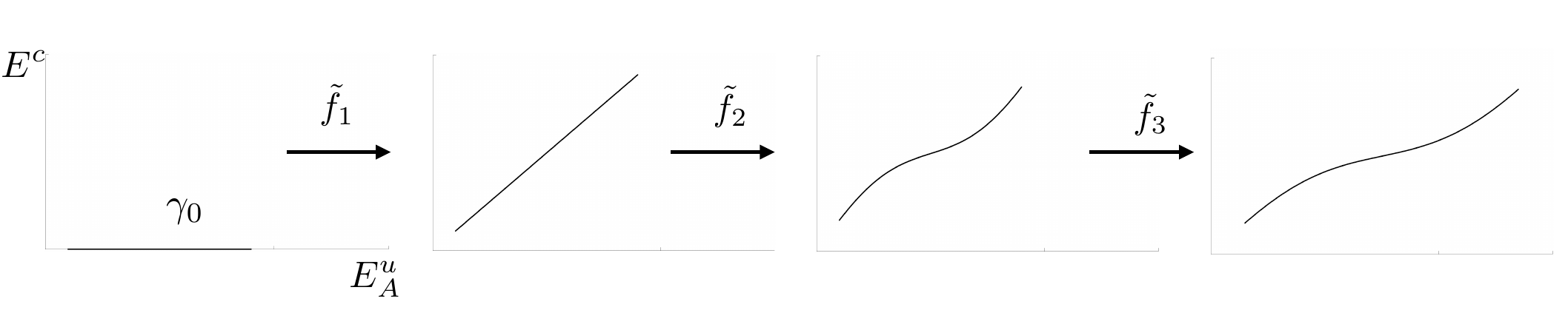}

\medskip
\includegraphics[page=3,scale=0.35]{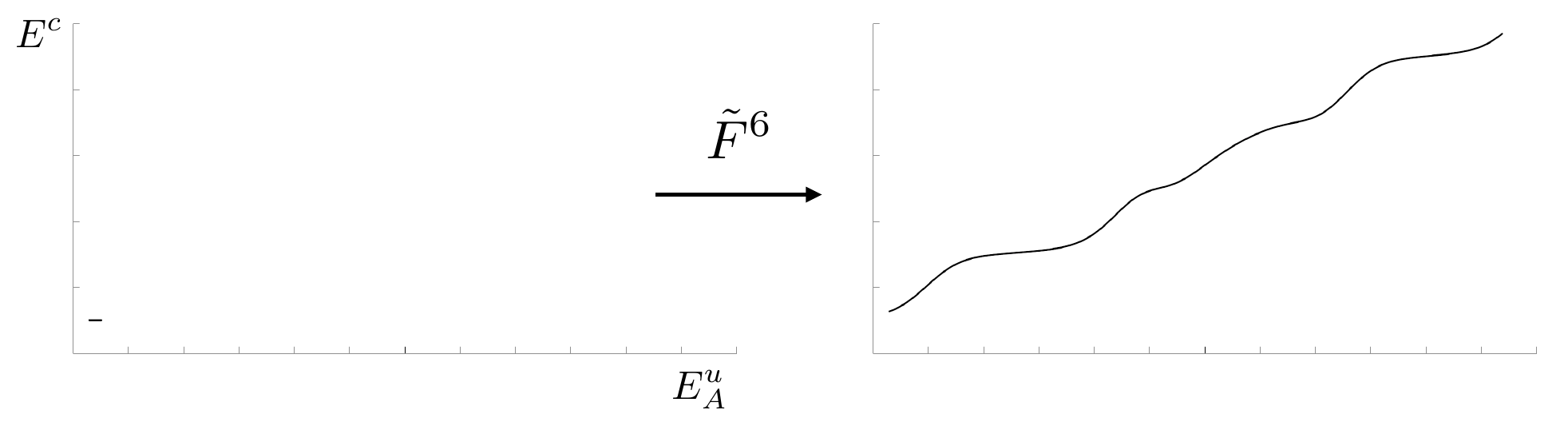}
\caption{Top: Evolution of a curve in $W^{cu}$. 
Here we show the action of the map on a curve $\gamma_0$, $\gamma_0(J) \subset \{z=\mbox{constant}\}$,  such that $\pi_{xy}\gamma_0$ is contained in a $W^u_A-$curve. The pictures on the top show  the effect of $\tilde f_1$, $\tilde f_2$, and $\tilde f_3$ on $\gamma_0$, where $\tilde f_1$, $\tilde f_2$, and $\tilde f_3$ are the lifts of $f_1$, $f_2$, and $f_3$ respectively. Here we assume that $r(x)=x$. Bottom: action of $\tilde F^6$ on a piece of curve $\gamma_0$, where $\tilde F$ is the lift of the map $F$. These plots highlight the effect of the monotonicity of the coupling.}
\label{Fig:AdmissibleLeaves}
\end{figure}

\subsubsection{ Persistence of splitting  and foliations.} \label{Rem:PersUnderSkew}
Above we showed that every skew-product map $F$ satisfying
Conditions (A1)-(A3) has a continuous splitting $E^u \oplus E^c \oplus E^s$ defined everywhere
on $\T^2 \times \T$, and that with respect to this splitting, $DF$ satisfies a uniform domination condition. By standard invariant cones arguments, these properties are passed 
(with slightly relaxed bounds) to all maps $C^1$-near $F$. Furthermore, the foliation  of $\T^2 \times \T$ by circle fibers is smooth and  by Condition (A1), also
 normally hyperbolic,  i.e. each leaf $(x,y)\times \T$ is tangent to $E^c$ and for any $p\in (x,y)\times \T$ 
 \[
 \|DF_p|_{E^c}\|<\|DF_p(\partial_z)\|<\min_{\substack{v\in E^u,}{\|v\|=1}}\|DF_pv\|.
 \]
    By  Theorem 7.1 in \cite{hirsch1977invariant}, such a foliation, whose leaves are closed curves tangent to $E^c$,  persists under
small $C^1$ perturbations, although in general the perturbed foliation is only continuous. That is to say, if $F$ is
a skew-product map, $\mathcal F$ is the foliation of 
$\T^2 \times \T$ into circle fibers, and
$G$ is a $C^1$-small perturbation of $F$, then there is a unique $G-$invariant continuous foliation $\mathcal G$ and a homeomorphism $H$ of
$ \T^2 \times \T$ to itself that is a leaf-conjugacy, meaning $H$
carries the leaves of $\mathcal F$ to the leaves of $ {\mathcal G}$, and
passes to a homeomorphism $H/\sim$ from
$(\T^2 \times \T)/\mathcal F$ to $(\T^2 \times \T)/{\mathcal G}$ that
conjugates the two quotient dynamical systems.


\subsection{Contracting centers}\label{Sec:CondDensity}

\ \medskip

 Returning to the skew-product map $F$, this subsection studies conditions under which $DF|_{E^c}$ is,  {\it on average},
contracting. Such a notion requires that we specify a reference measure on unstable curves.
With an eye toward SRB measures,  the following is a natural choice; see e.g.
\cite{hasselblatt2002handbook}.
For any segment  of $W^u$-curve $W$,   we let $m_W$ denote the arclength measure on $W$, and
let $\rho_W:W \to \mathbb R$ be the function with the property that for all $p, p' \in W$,
\begin{equation} \label{density}
\frac{\rho_W(p)}{\rho_W(p')} \ = \ 
\lim_{n \to \infty} \frac{DF^n_{F^{-n}p'}|_{E^u}}{DF^n_{F^{-n}p}|_{E^u}} \ = \ 
\lim_{n \to \infty} \prod_{i=1}^n \frac{DF_{F^{-i}p'}|_{E^u}}{DF_{F^{-i}p}|_{E^u}}\ ,
\end{equation}
normalized so $\int_W \rho_W \ dm_W = 1$. Because distances on $W^u$ contract exponentially
fast in backward time, the limit above exists and convergence is exponential.
Let $\mu_W$ be the probability measure defined by $d\mu_W = \rho_W dm_W$.
Two useful properties of $\mu_W$ that can be deduced from the definition of $\rho_W$ are  (i)  $F_*\mu_W=\mu_{FW}$ where $F_*\mu_W$
is the pushforward of the measure $\mu_W$ by $F$,
and (ii) given $\ell>0$,
there is $K=K(\ell)>0$ such that for any $W$ of length $\le \ell$ and any $p,p'\in W$
\begin{equation}\label{Eq:bndDist}
\frac{\rho_W(p)}{\rho_W(p')}\le K\ .
\end{equation}
The distortion constant $K$ depends on $\alpha^u_{min}$ and on 
the $C^2$ norm of $F$.

\begin{definition} We say $F$ satisfies a {\it contracting center condition} (with respect to the
reference measures $\mu_V$)
if there exists $ c_1>0$ such that for any fundamental domain $V$
\begin{equation} \label{CC}
\int_{F^{-1}V} \log DF_p|_{E^c} \ d\mu_{F^{-1}V}(p) \ \le \ - c_1\ .
\end{equation}
\end{definition}

The above condition, which we will henceforth abbreviate as the ``CC-condition", is reminiscent 
of similar conditions in the literature (see e.g. \cite{dolgopyat2000dynamics}). It is natural to
integrate over $F^{-1}V$ (rather than $V$) because the action of $DF|_{E^c}$ on $F^{-1}V$
translates into the action of $g'$ on the full circle; see (\ref{DFc}).

\medskip
Below we give a condition in terms of the various derivative bounds of $F$ (see Sect.~\ref{Sec:AssumDom})
that implies the CC-condition.
Let $I^-:=\{z\in\T:\,g'(z)< 1\}$ and $I^+:=\{z\in\T:\,g'(z)\ge 1\}$. If $V_A$ is a segment of unstable leaf for the Anosov $A$, one can define $\rho_{V_A}$ as in \eqref{density}  above by substituting $F$ with $A$.  Analogous properties hold 
for $\rho_{V_A}$; we denote the distortion constant for $\rho_{V_A}$ by $K_A$.

\begin{proposition}
Asssume (A1)-(A3). If
\begin{equation}\label{Eq:NegExpCond}
 \left[\frac{CK_A^2}{c\beta}\left( 1+\frac{2\lambda_g^{max} }{\alpha^u_{min} - \lambda_g^{max}}\right)\right]\, \int_{I^+}\log g'(s)ds+ \int_{I^-}\log g'(s)ds<0\ ,
\end{equation}
 then $F$  satisfies the CC-condition.
 \end{proposition} 

\begin{proof}
 Let $V$ be a segment of $W^u$-curve of $F$ that is a fundamental domain, and let 
$V_A:=\pi_{xy}(V)$. Notice that $(\pi_{xy})_*\mu_{V}=\mu_{V_A}$, and this implies $\mu_{F^{-1}V}=(\pi_{xy}|_{F^{-1}V})^{-1}_*\mu_{A^{-1}V_A}$. Therefore,
\[
\int_{F^{-1}V} \log DF_{p}|_{E^c} d\mu_{F^{-1}V}(p)
\ = \ \int_{\T} \log g'(s)\, d((\pi_z \circ f_1 \circ (\pi_{xy}|_{F^{-1}V})^{-1})_*\mu_{A^{-1}V_A})(s)
\]
where $f_1$ is as defined in Sect. 4.2. 

 By the change of variables formula 
\[
\frac{d(\pi_z \circ f_1 \circ (\pi_{xy}|_{F^{-1}V})^{-1})_*\mu_{A^{-1}V_A}}{ds}(s)=\rho_{A^{-1}V_A}(p(s))\,  D(\pi_{z}\circ f_1\circ \pi_{xy}^{-1}|_{A^{-1}V_A})^{-1}_s,
\] 
where $p(s):=[\pi_z \circ f_1 \circ (\pi_{xy}|_{F^{-1}V})^{-1}]^{-1}(s)$. Recall that the tangent direction to $F^{-1}V$ is $E^u\in \mc C^u_{b,a}$ which, together with assumption (A3)  implies that the tangent direction to $f_1(F^{-1}(V))$ is in $\mc C^u_{b+c\beta,a+C}$, and therefore $b+c\beta\le D(\pi_{z}\circ f_1\circ \pi_{xy}^{-1}|_{A^{-1}V_A})\le a+C $. Recalling that $K_A^{-1}\le \rho_{A^{-1}V_A}\le K_A$ we obtain
\[
[(a+C)K_A]^{-1} \le\frac{d(\pi_z \circ f_1 \circ (\pi_{xy}|_{F^{-1}V})^{-1})_*\mu_{A^{-1}V_A}}{ds}(s)\le \frac{K_A}{c\beta}.
\]
Therefore
\begin{align*}
\int_{F^{-1}V} &\log DF_{p}|_{E^c} d\mu_{F^{-1}V}(p)\\
&\le \frac{K_A}{c\beta}  \int_{I^+}\log g'(s)ds+\left[\left( \frac{2\lambda_g^{max} }{\alpha^u_{min} - \lambda_g^{max}}+1\right)CK_A\right]^{-1} \int_{I^-}\log g'(s)ds
\end{align*}
and the above is strictly less than zero if condition \eqref{Eq:NegExpCond} is satisfied.
\end{proof}

 The following is a concrete example of a map $F$ satisfying the CC-condition. Pick any $g$ satisfying  the conditions in Sect. 2.1.
This condition implies that $g' $ is not constant. By Jensen's inequality,
\[
\int_{\T} \log g'(s) ds < \log \int_{\T} g'(s) ds = 0.
\]
 So, there exists $\delta>0$ such that 
\[
(1+\delta)  \int_{I^+}\log g'(s)ds+ \int_{I^-}\log g'(s)ds<0.
\]
Now pick $A$ to be any linear Anosov diffeomorphism such that 
$\beta=\partial_A^u \cdot \partial_x$ satisfies $\beta^{-1} < (1+\frac12 \delta)$.
Pick $r$ also linear, i.e. $r'=1$. Under these assumptions $K=c=C=1$ and condition \eqref{Eq:NegExpCond} reads
\[
\frac{1}{\beta} \left( 1+\frac{2\lambda_g^{max} }{\alpha^u - \lambda_g^{max}}\right) \int_{I^+}\log g'(s)ds+\int_{I^-}\log g'(s)ds<0
\]
where $\alpha^u$ is the uniform expansion rate of $A$. 
 The quantity in brackets on the left can be made arbitrarily close to $1$
when $\alpha^u$ is sufficiently large relative to $\lambda_g^{max}$. 
 It remains to ensure that $\beta$ is close to $1$. For that one can choose, e.g.,
\[
A=\left(\begin{array}{cc}
N& N-1\\
1 & 1
\end{array}\right)
\]
for which the expanding direction is parallel to  
\[
 \binom {{N-1+\sqrt{(N-1)(N+3)}}}{2},
\]
and is therefore  as aligned with the $x$-axis as we wish for $N$ sufficiently large.  Notice that by picking $A$ with $\beta$  close to one and $\alpha^u$ large, we have also satisfied assumptions (A1)-(A3).

\section{ Regular Monotonic Interactions: Statistical Properties}\label{Sec:RegMonotStatist}
The main result of this
section, namely the existence of SRB measures for an open set of nonuniformly
hyperbolic systems  obtained by regular, monotonic couplings of 
$A$ and $g$,
is stated and proved in 
Sect. \ref{Sec:OpenSetSRB}. Uniqueness of SRB measure 
is proved in Sect. \ref{Sec:Uniqueness} under additional restrictions. 
Sect. \ref{Sec:SRBandphys} contains a brief review of SRB measures and 
Sect. \ref{Sec:LiterRev} discusses connections to the existing literature.

\subsection{SRB and physical measures}\label{Sec:SRBandphys}

\ \medskip

We provide here a brief review of SRB measures and physical 
measures; see  \cite{eckmann1985ergodic} for more information.

Let $\Phi: M \circlearrowleft$ be a $C^2$ diffeomorphism of a compact Riemannian
manifold $M$. A $\Phi$-invariant Borel probability measure $\mu$ is called an 
{\it SRB measure} if (i) $\Phi$ has a positive Lyapunov exponent $\mu$-a.e., and
(ii) the conditional measures of $\mu$ on the unstable manifolds of $\Phi$ are
absolutely continuous with respect to the Riemannian volume on these manifolds.

The following is one of the reasons why SRB measures are important. 
Equating observable events with sets of positive Lebesgue (or 
Riemannian) measure, we call
an invariant Borel probability measure $\mu$ a {\it physical measure} if there is a positive
Lebesgue measure set $\mathcal B(\mu) \subset M$ with the property that for any 
continuous function $h:M \to \mathbb R$, 
\begin{equation} \label{physical}
\lim_{n\rightarrow \infty }\frac1n \sum_{i=0}^{n-1} h(\Phi^i x) \ = \ \int_M h d\mu
\qquad \mbox{for every } x \in \mathcal B(\mu)\ .
\end{equation}
This is not the Birkhoff Ergodic Theorem, as $\mu$ need not be absolutely continuous 
with respect to Lebesgue measure. By the absolute continuity of stable foliations,
 ergodic SRB measures with no zero exponents are physical measures
\cite{pugh1989ergodic}, and any SRB measure with no zero
Lyapunov exponents can be decomposed into at most 
a countable number of ergodic components each one of which is an SRB measure.

The significance of SRB measures was recognized by Sinai, Ruelle and Bowen, 
who constructed these measures for Axiom A attractors
in the 1970s \cite{sinai1972gibbs, ruelle1976measure}; see also \cite{bowen1975equilibrium}. 
The  concept was extended to more general dynamical systems by Ledrappier,
Young and others in the 1980s \cite{ledrappier1985metric,ledrappier1984proprietes,young1985bowen,ledrappier1981proof}. 

Not all attractors admit SRB measures, however. Outside of the 
Axiom A category there are not many concrete examples; see e.g. \cite{young1998statistical,young2002srb}.
The attractors discussed in this paper are not Axiom A, but
they are not far from Axiom A. They belong in a class of dynamical systems called 
(uniformly) partially hyperbolic systems. See Sect. \ref{Sec:LiterRev} below for  
a more detailed discussion.



\subsection{Open sets of nonuniformly hyperbolic maps with SRB measures}\label{Sec:OpenSetSRB}

\ \medskip

Let Diff$^2(\T^2 \times \T)$ denote the set of $C^2$ diffeomorphisms of $\T^2 \times \T$ 
onto itself equipped with the $C^2$ topology.  The following is the
main result of this section.

\begin{theorem}\label{Thm:OpenDenseSet} 
Let $ F \in \mbox{Diff}^2(\T^2 \times \T)$ be the skew-product map given by Eq. \eqref{Eq:SkewF}. 
We assume it satisfies (A1)-(A3)  in Sect. \ref{Sec:AssumDom} and the CC-condition in Sect. \ref{Sec:CondDensity}. Then there is a neighborhood $\mc N$ of $ F$ in Diff$^2(\T^2 \times \T)$ 
such that every $G\in \mc N$ admits an SRB measure $\mu_G$ with one positive and two
negative Lyapunov exponents.
\end{theorem}

\smallskip
 It follows that all ergodic components of $\mu_G$ are physical measures.

As to whether the SRB measures in Theorem \ref{Thm:OpenDenseSet} are supported on Axiom A attractors,
a straightforward extension of Proposition \ref{Prop:DensityLyap} to maps near $F$ that
are not necessarily skew products together with the examples in Sect. \ref{Sec:CondDensity}  gives the following result.

\begin{corollary}
 Among the maps obtained by coupling together $A$ and $g$, there exist 
$C^2$-open sets $\mathcal N$ with the properties that 

(i) no $G \in \mathcal N$ has an Axiom A attractor, and

(ii) all $G \in \mathcal N$ admit SRB measures.
\end{corollary}

\medskip \noindent
{\bf Remark on Terminology:} 
 We have used $W^u, W^{cu}, W^{cs}$ and $W^s$ to denote 
the invariant manifolds
tangent to the subbundles $E^u, E^c \oplus E^u, E^c \oplus E^s$ and $E^s$ 
respectively. This introduces a slight conflict with the usual nomenclature of stable and 
unstable manifolds. For example, if $DF|_{E^c}$ is in fact contracting, 
as in the case where
the CC-condition is satisfied, then our $W^{cs}$-manifolds are what is usually referred 
to as stable manifolds, and our $W^{s}$-manifolds are what is usually called strong
stable manifolds. In the interest of notational consistency among the different sections
within this paper, we will adhere to the notation that $W^*$ denotes the 
invariant manifolds tangent to $E^*$ for $* = u, cu, s$, and $cs$; but
to avoid confusion, we will refrain from using the terms ``stable manifolds"
or ``unstable manifolds" in technical proofs.

\medskip

\begin{proof}[Proof of Theorem \ref{Thm:OpenDenseSet}] 
  Let $F$ be as in Theorem \ref{Thm:OpenDenseSet}. 
Recall from Sect. \ref{Rem:PersUnderSkew} that any sufficiently small $C^1$-perturbation $G$ of $F$ inherits
a dominated splitting $E^s\oplus E^c \oplus E^u$. We note also that the definitions of
fundamental domains and the CC-condition (as defined in Sects. \ref{Sec:MixedBehav} and \ref{Sec:CondDensity} for 
skew-product maps) carry over to $G$ provided $G$ is $C^2$ close to $F$. 
Steps 1 and 2 below prove the existence of an SRB measure for $G$ assuming 
it satisfies the CC-condition.
Step 3 justifies the CC-condition for all $G$ sufficiently close to $F$ in the $C^2$-norm.

\medskip \noindent
 {\it 1. Construction of invariant probability measures with absolutely continuous
conditional measures on $W^u$-leaves.} We follow the standard construction of SRB measures 
for Axiom A attractors in e.g. \cite{Youngergodic}: 
Let $W$ be any finite segment of $W^u$-leaf for $G$,
and let  $m_W$ be the arclength measure on $W$. Letting
\[
\mu_n \ := \ \frac1n \sum_{i=0}^{n-1} G^n_*(m_{W})\ ,
\] 
we are assured that a subsequence of $\mu_n$ will converge weakly, and any limit point 
$\mu_\infty$ is easily shown to have absolutely continuous conditional measures on  $W^u$-leaves.
Moreover, if $\eta$ is a partition of $\T^2 \times \T$ whose elements
are unstable curves of
finite length, then the condition probability densities of $\mu_\infty$ on any $W\in\eta$ is 
precisely $\rho_W$, the reference measures  introduced in Sect. \ref{Sec:CondDensity}.

\medskip \noindent
{\it 2. Proof of SRB  property assuming the CC-condition.}
We now
prove that any $\mu_\infty$ constructed in Item 1 above is an SRB measure assuming that 
$G$ satisfies the CC-condition.  To do that, it suffices to show that $\mu_\infty$-a.e., 
the Lyapunov exponent in the $E^c$-direction is strictly negative.

Observe first that almost every ergodic component of $\mu_\infty$ has conditional densities
on $W^u$-leaves. This is because all points on a $W^u$-leaf have the same 
asymptotic distribution in backwards time, so they cannot be generic with respect to distinct
ergodic measures. 

Now let $\eta$ be a measurable partition of $M=\T^2 \times \T$ with the property that the
elements $W$ of $\eta$ are fundamental domains of $W^u$-leaves (for instance we can pick $\pi_z(\partial W) = \{z_a\}$). Let $\mu$ be an ergodic component of $\mu_\infty$,
 and let $\{\mu_V, V \in G^{-1}\eta\}$ be a regular family of conditional probability
measures on $G^{-1}\eta$.\footnote{This is a benign abuse of notation: 
In Sect. \ref{Sec:CondDensity} we used $\mu_V$ to denote a class of reference measures on $W^u$-leaves, and here
we use the same notation to denote conditional probability measures of the SRB measure $\mu$. These two usages in fact coincide; that was the motivation for the choice of reference measure in Sect. \ref{Sec:CondDensity} to begin with.}
Letting
$\mu/G^{-1}\eta$ denote the quotient measure of $\mu$ on $M/G^{-1}\eta$, we have
\[
\int \log DG_p|_{E^c} \ d\mu (p) = 
\int_{M/G^{-1}\eta} d(\mu/G^{-1}\eta)(V) \int_V \log D G_p|_{E^c} \ d\mu_V(p)\ .
\]
The CC-condition says precisely that for each $V \in G^{-1} \eta$, the $d\mu_{V}$-integral 
 is strictly negative. Thus the integral on the left is strictly negative, implying, by ergodicity of $\mu$, that the Lyapunov
exponent in the $E^c$-direction is strictly negative $\mu$-a.e. Thus it holds $\mu_\infty$-a.e.
 since it holds for every ergodic component $\mu$ of $\mu_\infty$.

\medskip \noindent
{\it 3. Openness of the CC-condition.} 
We will show that if $F$ satisfies the CC-condition with constant $c_1$ in Eq. (\ref{CC}), 
then  for small enough $\delta>0$,  every $G\in\mbox{Diff}^2(\T^2 \times \T)$  
with $d_{C^2}(G,F)<\delta$ will satisfy the CC-condition with constant $\frac12 c_1$.
As we will be comparing estimates for $G$ to those of $F$, 
let us agree to use ordinary notation for $G$ and put a bar above quantities associated with $F$.

Let $ V$ a fundamental domain for $G$.   Then there is a fundamental domain $\bar V$ of $F$ 
that can be made arbitrarily close to $ V$ in the sense that there is a $C^2$ embedding 
$\iota: V\rightarrow \T^2\times \T$ such that $\iota( V)=\bar V$ satisfying the following conditions
as $d_{C^2}(G,F) \to 0$: 
 \begin{align*}
\mbox{(a)}\quad\quad\quad & d_{C^1}(\iota, \Id|_{ V}) \to 0 ,\\
\mbox{(b)}\quad\quad\quad & \sup_{p\in G^{-1} V} \left|DG_{p}|_{ E^c}-D F_{\iota(p)}|_{ \bar E^c}\right| \ \to \ 0\ ,\\
\mbox{(c)}	\quad\quad\quad& 	\sup_{p\in G^{-1}V} \left|\rho_{G^{-1} V}(p)- \bar \rho_{ F^{-1}\bar V}(\iota(p))\right| \ \to \ 0 .
\end{align*} 
\medskip \noindent
The assertions in (a) and (b) follow from  Sect. \ref{Rem:PersUnderSkew}, and the assertion in (c) 
follows from the fact that the distortion estimate in (\ref{density}) converges exponentially 
fast in $n$ so it suffices to control them for a finite number of iterates. More precisely, for any 
$\epsilon>0$  there is $N\in \N$ such that for every  fundamental domain $V$ and any $p,p'\in G^{-1} V$, 
 \[
 \frac{\rho_{{G}^{-1} V}(p)}{\rho_{{G}^{-1} V}(p')}\le (1+\mc O(\epsilon))\prod_{i=1}^{N}\frac{(D{G}_{{G}^{-i}p'})|_{ E^u}}{(D{G}_{{G}^{-i}p})|_{ E^u}},
 \]
and an analogous inequality holds for fundamental domains of $F$.
Moreover, control of the differences in all three items (a), (b), and (c) above is uniform in $V$
and for all $G$ with $d_{C^2}(G,F)<\delta$ for small enough $\delta$.
 \end{proof}

\subsection{Uniqueness of SRB measures}\label{Sec:Uniqueness}

\ \medskip
 The setting is as in Theorem 5.1. Our objective here is to investigate
the uniqueness of the SRB measure for $G \in \mathcal N$.
Assume without loss of generality that the Anosov diffeomorphism $A$ has a fixed
point at $(0,0)$. Then $F$ leaves invariant the circle fiber $\{(0,0)\} \times \T$.
By the persistence of normally hyperbolic manifolds (see (A1)), every $G \in \mathcal N$
leaves invariant a closed curve near $\{(0,0)\} \times \T$; we will call it $\T_0$. 


\begin{lemma} For $G \in \mathcal N$, define
\[
W^u(\T_0) := \cup_{p \in \T_0} W^u(p) \qquad \mbox{and} \qquad
W^s(\T_0) := \cup_{p \in \T_0} W^s(p)\ . 
\]
Then (i) $W^u(\T_0)$ and $W^s(\T_0)$ are dense in $\T^2 \times \T$, and

\quad (ii) every $W^u$-curve meets $W^s(\T_0)$.
\end{lemma}

\begin{proof} (i) We prove the claim for $W^s(\T_0)$; the proof for $W^u(\T_0)$ is analogous.

As $\T_0$ is invariant and normally hyperbolic, $W^{s}(\T_0)$ is an immersed submanifold; see Theorem 4.1 of \cite{hirsch1977invariant}. We know also that $W^{s}(\T_0)$ is tangent to $E^{s}\oplus E^{c}$
because for every $q \in W^{s}(\T_0)$, there is $p \in \T_0$ such that 
that $d(F^np, F^nq)\to 0$ exponentially fast, so $q$ must have the same Lyapunov
exponents as $p$. 

{}From Sect. 4.2.3, we know that there is a $G$-invariant foliation $\mathcal G$
the leaves of which are closed curves tangent to $E^c$. This implies that $W^{s}(\T_0)$ 
is foliated by the leaves of $\mathcal G$, and in fact that $W^{s}(\T_0)$ is the union of 
$\mathcal G$-leaves 
that under iterates of $G$ tend to $\T_0$ as $n \to \infty$. Now it is also known that 
there is a homeomorphism $H: \T^2 \times \T \to \T^2 \times \T$ that is a leaf-conjugacy 
between $F$ and $G$, meaning it sends vertical fibers of the skew product map $F$
to the leaves of $\mathcal G$. Calling $\bar \T_0$ and $\bar W^s(\bar \T_0)$ the corresponding 
objects for $F$, we have, from the characterization of $W^s(\T_0)$ above, that
$H(\bar \T_0)=\T_0$ and $H(\bar W^{s}(\bar \T_0))=W^{s}(\T_0)$.
Clearly,  $\bar W^s(\bar \T_0)=W_A^s(0,0)\times \T$ is dense in $\T^2 \times \T$ as the
stable manifold $ W_A^s(0,0)$ of $A$ is dense in $\T^2$. Therefore $W^{s}(\T_0)$ is also
dense since $H$ is a homeomorphism.

This completes the proof of (i).

\medskip \noindent 
(ii) follows immediately from the fact that $W^s(\T_0)$ is a densely immersed submanifold tangent
to $E^c \oplus E^s$.
\end{proof}


\smallskip
In the proofs of the results to follow, our strategy is to connect the dynamics of 
$G|_{\T_0}$ to those of $G$ by

-- using $W^s(\T_0)$ to draw relevant sets close to $\T_0$, 

-- using $G|_{\T_0}$ to transport the sets around this fixed fiber and then 

-- using $W^u(\T_0)$ to deliver them to where we would like them to go.

\medskip
Let $\mu$ be an ergodic SRB measure. We introduce the following definition for use below: We call a piece 
of $W^u$-leaf $W$ ``$\mu$-typical" if the following hold at $m_W$-a.e. 
$p$ (recall that $m_W$ is the arclength measure): 

(i) $p$ is generic with respect to $\mu$, i.e., for all continuous
$h: \T^2 \times \T \to \R$, 
\begin{equation} \label{generic}
\lim_{n\rightarrow \infty}\frac{1}{n} \sum_{i=0}^{n-1} h(G^ip) = \int_{\T^2\times \T} h d\mu  ;
\end{equation}

(ii)  $G$ has two strictly negative Lyapunov exponents at $p$ and 
$W^{cs}_{\rm loc}(p)$ is well defined with the property that for every $q\in W^{cs}_{\rm loc}(p)$, $d(G^{n}(p),G^n(q))\rightarrow 0$ as $n\rightarrow\infty$.

\medskip \noindent
 Let $W^*$ denote the set of all $p \in W$ with the properties in (i) and (ii).
By the absolute continuity of $W^{cs}$-manifolds  together with property (ii),
\[
W^{cs}_{\rm loc}(W^*) :=  \cup_{p \in W^*} W^{cs}_{\rm loc}(p)
\]
has positive  Lebesgue measure on $\T^2\times \T$, and all
$q \in W^{cs}_{\rm loc}(W^*)$ are generic with respect to $\mu$.

\smallskip
\begin{proposition} If $G \in \mathcal N$ is such that $G|_{\T_0}$ 
is conjugate to an irrational rotation, then it has a unique SRB measure.
\end{proposition}

\begin{proof}
Let $\mu_1$ and $\mu_2$ be ergodic SRB measures, and fix $W_1$, a $\mu_1$-typical 
$W^u$-curve of finite length. For $\varepsilon>0$, let 
$\mathcal U_\varepsilon(W_1)$ be  the tubular neighhorbood of $W_1$ given by
 $\mathcal U_\varepsilon(W_1) = \cup_{p \in W_1} \mbox{exp}_p B^{cs}_\varepsilon(p)$
 where 
$B^{cs}_\varepsilon(p)$ is the $\varepsilon$-ball in $E^{cs}(p)$ centered at $p$, and 
 $\exp_p$ is the exponential map at $p$.
We claim that there exists $\varepsilon=\varepsilon(W_1)>0$
such that  if there is a $\mu_2$-typical $W^u$-curve $W_2$ roughly
parallel to $W_1$ that traverses the length of $\mathcal U_\varepsilon(W_1)$, 
then $\mu_2=\mu_1$. This is true because for 
$\varepsilon$ small enough, there is  $\Gamma\subset W^*_1$ with  
$m_{W_1}(\Gamma)>0$ such that  for all $p \in \Gamma$, $W^{cs}_{2\varepsilon}(p)$ is well defined
and contains a full cross-section of $\mathcal U_\varepsilon(W_1)$. By the absolute continuity
of the $W^{cs}$-foliation, $W_2$ 
will meet $\cup_{p \in \Gamma} W^{cs}_{2\varepsilon}(p)$ in a positive $m_{W_2}$-measure set,
and all points in this set are generic with respect to $\mu_1$.

Since $W^u(\T_0)$ is dense in $\T^2 \times \T$  (Lemma 5.3 (i)),
there exist $p' \in \T_{0}$ and a finite-length segment 
$V'$ of $W^u(p')$ containing $p'$ that crosses the length 
of $\mathcal U_{\frac14 \varepsilon}(W_1)$. 
It follows that for $\varepsilon'>0$ small enough, for all 
$p'' \in \T^2 \times \T$ with $|p''-p'|< \varepsilon'$, 
a subsegment $V''$ of $W^u(p'')$ near $V'$
will cross $\mathcal U_{\frac12 \varepsilon}(W_1)$. By the topological
transitivity of $G|_{\T_0}: \T_0 \circlearrowleft$, there exists $\varepsilon''>0$
such that any segment of $W^u$-curve meeting any part of the $\varepsilon''$-neighborhood of
$\T_0$ will, when iterated forward, be transported near $p'$ and eventually  
cross $\mathcal U_\varepsilon(W_1)$.

Let $W_2$ be any $\mu_2$-typical curve.
 By Lemma 5.3(ii), $W_2$ meets $W^s(q)$ for some $q \in \T_0$.
Iterating forward, $W_2$ is brought
as close to $\T_0$ as we wish, and by the argument above, a future iterate of
$W_2$ -- which is also $\mu_2$-typical -- can be made to cross 
$\mathcal U_\varepsilon(W_1)$ implying $\mu_2=\mu_1$.
\end{proof}

Recall that in the space of $C^2$ diffeomorphisms of the circle $\T$, an open and
dense set consists of maps with a finite number of sinks and sources, and with
the property that every
orbit approaches a sink in forward time and a source in backward time. 
For $k=1,2,\cdots$, let $\mathcal N_k$ be the subset of $\mathcal N$ for which 
$F|_{\T_0}$ has the property above with exactly $k$ attractive periodic orbits.
 It is easy to see that $\cup_k \mathcal N_k$ is an open and dense 
subset of $\mathcal N$.

\smallskip
\begin{theorem} For each $k \ge 1$, there is an open and dense subset  
$\mathcal N_k' \subset \mathcal N_k$ with the property that all $G \in 
\mathcal N'_k$ have $\le k$ ergodic SRB measures. In particular, 
uniqueness of SRB measures is enjoyed by all $G \in \mathcal N'_1$. 
\end{theorem}

\begin{proof} We treat first the case $k=1$, assuming that the attractive periodic
cycle is a fixed point, i.e., we assume $G|_{\T_{\bf 0}}$ has a repelling fixed 
point at ${\bf 0}_r \in \T^2 \times \T$, an attractive fixed point at ${\bf 0}_a$, 
and all points  in $\T_0$ other 
than ${\bf 0}_r$ are attracted to ${\bf 0}_a$. Generalization to the periodic case
is straightforward and left to the reader. 

Let $G \in \mathcal N_1$. We construct an SRB measure  by
pushing forward  Lebesgue measure on $W^u_{\rm loc}({\bf 0}_a)$. 
Let $\mu_1$ be an ergodic component of the measure constructed. 
Let $W_1$ be a $\mu_1$-typical $W^u$-curve,
and let $\mathcal U_\varepsilon(W_1)$ be as in Proposition 5.3.
By construction, there is a segment $V \subset W^u({\bf 0}_a)$
such that $V$ crosses $\mathcal U_{\frac12 \varepsilon}(W_1)$.

Let $\mu_2$ be another ergodic SRB measure, and let $W_2$ be a $\mu_2$-typical 
$W^u$-leaf.  By Lemma 5.3(ii), $W_2$ meets $W^s(q)$ for some $q \in \T_0$.
 If $q \ne {\bf 0}_r$, then since $F^n q \to {\bf 0}_a$, $W_2$ will,
under forward iterates, be brought as close to ${\bf 0}_a$ as we wish, and a future image of it will
cross $\mathcal U_\varepsilon(W_1)$ proving $\mu_2=\mu_1$.

The only problematic scenario is when  $W_2 \cap W^s(\T_0 \setminus \{{\bf 0}_r\})
= \emptyset$ for all $\mu_2$-typical curves,
i.e., they meet only $W^s({\bf 0}_r)$. 
Observe that this can happen only when $W^u({\bf 0}_r) 
\cap W^s(\T_0 \setminus \{{\bf 0}_r\}) = \emptyset$, for otherwise 
$W_2$ will be brought near ${\bf 0}_r$ by $W^s({\bf 0}_r)$; from there 
it will follow $W^u({\bf 0}_r)$ and run into  $W^s(\T_0 \setminus \{{\bf 0}_r\})$ after all. 
It suffices therefore to show that with an arbitrarily small perturbation of $G$, one 
can cause $W^u({\bf 0}_r)$ to intersect $W^s(\T_0 \setminus \{{\bf 0}_r\})$. Such 
a condition is clearly open. This will be our definition of $\mathcal N'_1$.

A specific way to make such a perturbation is as follows: Let $\Gamma^u$ 
be a segment of $W^u({\bf 0}_r)$ containing ${\bf 0}_r$ of finite length. 
For $p \in \T^2 \times \T$, let $W^c(p)$ denote the circle fiber tangent to 
$E^c$ through $p$, and let 
$W^c(\Gamma^u) = \cup_{p \in \Gamma^u} W^c(p)$.
Now pick $\Gamma^s$, a segment of $W^s({\bf 0}_r)$ containing ${\bf 0}_r$
with the property that $W^c(\Gamma^u)$ and $W^c(\Gamma^s)$ meet in exactly
two $W^c$-fibers,  $\T_0$ and $W^c(p^*)$ for 
some $p^* \in \Gamma^u$. 
Since by assumption, $W^u({\bf 0}_r)$ does not meet 
$W^s(\T_0 \setminus \{{\bf 0}_r\})$, we must have  
$p^* \in \Gamma^s$. Perturb $G$ in a small neighborhood $U$
of $G^{-1}p^*$ to a map $G'$ so that $G'(G^{-1}p^*) \in 
W^s(\T_0 \setminus \{{\bf 0}_r\})$. Provided that $U$ is away from 
${\bf 0}_r$ and from $\Gamma^s$, we have arranged to have 
$W^u({\bf 0}_r; G')$ meet $W^s(\T_0 \setminus \{{\bf 0}_r; G')$
where $W^u(\cdot; G')$ and $W^s(\cdot; G')$ refer to objects associated
with the map $G'$.

This completes the proof of the $k=1$ case.

\medskip
For $k >1$, we sketch a proof again assuming (for simplicity of notation)
that the attractive periodic orbits are fixed points. Let 
${\bf 0}_{a,1}, \cdots, {\bf 0}_{a,k}$ be the attractive fixed points,
and for each ${\bf 0}_{a,i}$, construct an ergodic SRB measure $\mu_i$ 
by pushing forward Lebesgue measure on a segment of $W^u({\bf 0}_{a,i})$
as was done in the case $k=1$. Repeating the $k=1$ proof {\it verbatim}, 
we show that the $\mu_i$ are the only ergodic SRB measures.
We do not know, however, that the $\mu_i$ are necessarily distinct.
That is why we conclude only that the number of ergodic SRB measures
is $\le k$.
\end{proof}

\subsection{Relation to the existing literature} \label{Sec:LiterRev}

\ \medskip

We recapitulate the situation and discuss our results as they relate
to the existing literature. For all diffeomorphisms of $\T^2 \times \T$ that are
$C^1$-small perturbations of the skew product maps of the form in (\ref{Eq:SkewF}),
we proved the existence of a splitting $E^u \oplus E^c \oplus E^s$ with a uniform
domination condition.   General theories of hyperbolic systems with 
domination conditions have been 
studied, among others, in \cite{brin1973partially,diaz1999partial, bochi2003lyapunov, bonatti2003c1, abdenur2006global, pujals2009dynamics} (see  \cite{hirsch1977invariant} for a more systematic treatment). Assuming additionally that the map is $C^2$, the existence
of what is called a $u$-Gibbs measure followed \cite{pesin1982gibbs}.
Call this measure $\mu$.
  We further identified a checkable condition under which
 we proved that dynamics in the center bundle is contracting $\mu$-a.e., so that
$\mu$ is in fact an SRB measure. Uniformly partially hyperbolic
systems with a contracting center have been studied, among others in \cite{bonatti2000srb, dolgopyat2000dynamics, bonatti2005transitive,dolgopyat2016geometric}.

What is different  in this paper
 is that we did not start from these dynamical hypotheses. We
coupled together two much studied dynamical systems in a particular way, 
deduced the dynamical picture above and proved the existence of SRB measures.
We further demonstrated  that
in our setting, while the dynamics in $E^c$ is contracting $\mu$-a.e.,
the set of points at which it is asymptotically expanding is dense, showing genuinely
mixed behavior in this ``neutral" direction.

These examples join the handful of concrete examples of nonuniformly hyperbolic systems
known to have SRB measures, including e.g. \cite{young1985bowen, benedicks1993sinai, wang2002invariant, lu2013strange}. The examples in this paper possess invariant cones
and are much closer to uniform hyperbolicity. They contribute nevertheless
to expanding the relatively small
group of dynamical systems known to have SRB measures.

Setups having similarities -- as well as differences -- to ours were considered in \cite{B18,H12,Y93}. In \cite{Y93}, the fiber maps come from the projectivization of a 2D matrix
cocycle. It is a special case of our skew-product maps; the domination condition is also different.
\cite{B18} is a generalization of \cite{Y93} to the skew-product setting; very high rates of expansion 
and contraction for $g$ are imposed.
\cite{H12} builds on results from \cite{ruelle2001absolutely}, and studies  volume preserving  perturbations of some skew-product systems (e.g. $A\times \Id_\T$). There volume-preservation
is assumed throughout, and it is used to study finer geometric properties of the system. 
Our setup is more flexible, and the existence of a physical measure is not {\it a priori} 
guaranteed. 

\section{Rare but Strong Interaction}\label{Sec:RareStrongInterac}

In this section, we continue to consider skew product maps of the form 
\begin{equation} \label{F4}
F(x,y,z) = (A(x,y), g(z+r(x))
\end{equation}
where $A: \T^2\circlearrowleft$ is an Anosov diffeomorphism, and $g$ and $r$ are
as before but $r$ is a degree-one map supported on a very small interval 
$I_\varepsilon$ of length $\varepsilon>0$. We call these interactions ``rare" because
a typical orbit of $A$ visits the strip $I_\varepsilon \times \T$ with frequency $\sim
\varepsilon$, and when it is outside of this strip, there is no interaction between the
base and the fiber map. In terms of its action on fibers, 
the deterministic dynamical system $F$ has the flavor of random rotations 
followed by long relaxation periods during which the orbit is attracted to 
$z_a \in \T$ (unless it is stuck at the other fixed point $z_r$), to be rotated again
by a random amount at a random time.

\subsection{ Assumptions and Results}\label{Sec:ResStrong}

\ \medskip

 Let $\partial^A_u(\cdot)$ denote the partial derivative of a function on $\T^2$ 
in the direction of $E^u_A$, in the direction where $x$ is increasing.
The following three conditions are assumed:

\bigskip \noindent
(B1) $\lambda_g^{max}>\max_p\|DA_p\|$

\medskip 
For each $\varepsilon>0$, 

\medskip \noindent
(B2)  {\it $r=r_\varepsilon:\T \to \T$ is a degree-one map 
with $r'|_{I_\epsilon} >0$ and $r|_{I_\epsilon^c}=0$ where $I_\varepsilon \subset \T$ 
is an open interval of length $\varepsilon$.}

\medskip \noindent
(B3) {\it viewing $r$ as a function on $\T^2$, we assume
$$
\partial_u^A(r) >  (2\varepsilon)^{-1} \quad \mbox{on } \T \setminus r^{-1} (-d, d) \quad \mbox{where} \quad d < \mbox{dist}(\partial I^+, \partial (g(I^+)))\ 
$$
where $I^+$ is defined below.}

\bigskip
Condition (B1) allows us to fix $c_1 \in (0,1)$, a neighborhood $I^+$ of the repelling fixed point 
$z_r$ of $g$, and $C_0 > 1$ such that 
(i) $g'\ge c_1^{-1}$ on $I^+$, and
(ii) $c_1^{-1} > C_0\max_p\|DA_p\|$.
This is the $I^+$ in (B3). 
Additionally we will fix $c_0<1$ and a neighborhood $I^- \subset \T$ of the attractive fixed point $z_a$ such that (i) $g' \le c_0$ on $I^-$ and (ii) $g(I^+) \cap I^- = \emptyset$. See Figure \ref{Fig:HypIllustr} below.

\begin{figure}[h!]
\center
\includegraphics[page=5, scale=0.5]{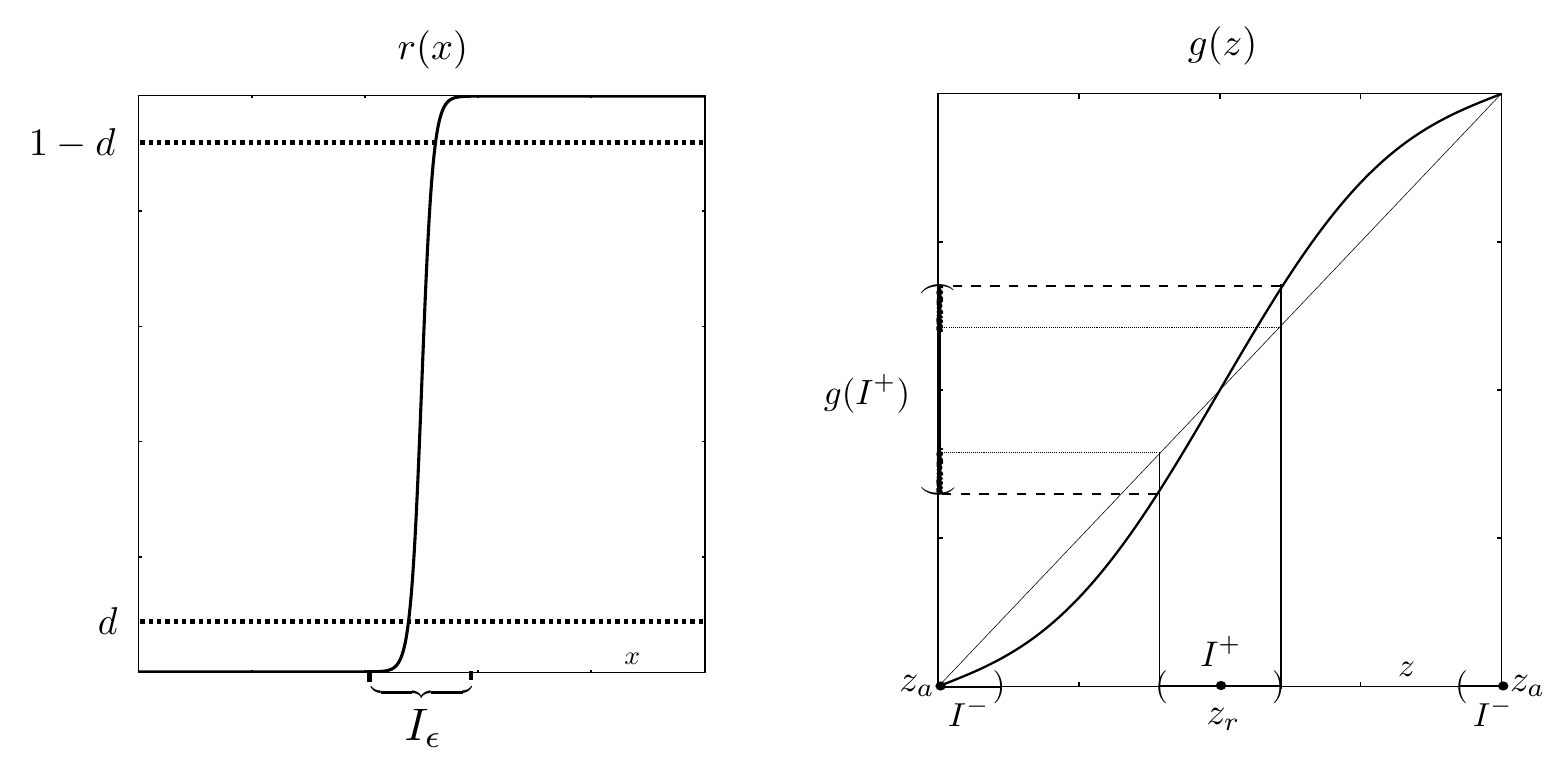}
\caption{The pictures above illustrate assumptions (B2) and (B3). On the left is the graph of an example of $r$: $r=0$ outside of $I_\epsilon$ and is very steep when $r(x)$ is in between $d$ and $1-d$. On the right is the graph of an example of $g$. The highlighted intervals are $g(I_+)\backslash I_+$; each component has diameter less than $d$. }
\label{Fig:HypIllustr}
\end{figure}

\bigskip
The main results of this section are as follows.

\begin{theorem} \label{Thm:6}
 Given $A$ and $g$ satisfying (B1), we assume for each $\varepsilon>0$
that $r=r_\varepsilon$ is chosen so (B2) and (B3) are satisfied.
 Then for all 
$\varepsilon>0$ small enough, $F$ has an SRB measure $\nu_\varepsilon$
with two negative Lyapunov exponents. Moreover, as $\varepsilon \to 0$, $\nu_\varepsilon$ converges weakly to $\mu \times \delta_{z_a}$
where $\mu$ is the SRB measure of $A$.
\end{theorem}

\smallskip

 The geometry of the maps satisfying (B1)-(B3) is somewhat more complicated than 
those studied earlier. The maps in Sects. \ref{Sect:RegMonotINtGeom} and \ref{Sec:RegMonotStatist} have a uniform domination condition on the entire phase space $\T^2 \times \T$.
Though there is ambiguous behavior in $E^c$, the directions corresponding to 
the strongest contraction and largest expansion are well separated. 
That is not so for the maps considered here: there are points whose strongest expansion 
occurs in the $z$-direction, and they are mixed with points with a negative
Lyapunov exponent in that direction. 

More precisely, consider the  subset $\Gamma_\varepsilon$ of the horizontal section
$\{z=z_r\}$ defined by
\[
\Gamma_\varepsilon := \{(x,y) \in \T^2: A^n(x,y) \not \in I_\varepsilon \mbox{ for all }n \ge 0\} 
\times \{z_r\}\ .
\]
Every $p \in \Gamma_\varepsilon$ has two expanding directions, the stronger expansion
occurring in the $z$-direction. This follows immediately from (B1) together with the fact that
$F(\Gamma_\varepsilon) \subset \Gamma_\varepsilon$.
Geometrically, $\Gamma_\varepsilon$ is a Cantor set together with its stable manifolds; 
for $\varepsilon$ small, this set is fairly dense in $\T^2\times\{z_r\}$. 
Theorem \ref{Thm:6}, on the other hand, asserts that on a set of positive 3D Lebesgue measure,
the Lyapunov exponent in the $z$-direction is negative. Moreover, as we will show,
 the unstable manifolds at $\nu_\varepsilon$-typical points  for an SRB measure $\nu_
\varepsilon$ crosses the horizontal section 
$\T^2\times \{z_r\}$ infinitely often.

 Because of the absence of a uniform domination condition on $\T^2 \times \T$,
we have less control on the geometry of the examples here than  in previous sections.
In particular, we do not know if the SRB property persists under small perturbations.

\subsection{Proof of Theorem \ref{Thm:6}}

\begin{proof}  We divide the proof into the following three steps.

\medskip \noindent
{\it 1. Construction of  invariant probability measures $\mu_\varepsilon$ with 
$\int \log |\partial_z F| d\mu_\varepsilon < 0$.}
Identifying $\T^2$ with $\T^2 \times \{0\}$, we let 
$\gamma_0$ be a segment of unstable manifold of $A$, 
and let $\gamma_n = F^n (\gamma_0)$. Let $\mu_0$ be the 
arclength measure on $\gamma_0$, normalized to have total measure $1$, and 
let $\mu_\varepsilon$ be any accumulation point  of $\frac1n \sum_{i=0}^{n-1} F^i_*\mu_0$. 

 We will show that for small  enough $\varepsilon$, $\mu_\varepsilon$
has the desired property.
 Let $f(x,y,z) = (x,y, z+r(x))$. We need to show that a large fraction of  the measure
$(\pi_z \circ f \circ F^n)_*\mu_0$ lies in $I^-$.
In the argument below we will make the simplifying assumption that $A$ is linear,
leaving it to the reader to insert the usual distortion estimates when it is not. 

The following notation will be useful: For each $n$, let $\theta_n: \pi_{xy}(\gamma_n) \to \T$
be such that the graph of $\theta_n$ is  $\gamma_n$, and let 
$\psi_n: \pi_{xy}(\gamma_n) \to \T$ be such that the graph of $\psi_n$ is $f(\gamma_n)$, 
i.e., $\psi_n = \theta_n + r$. 

At the $n$th step, let $\sigma_n \subset \gamma_n$ be a segment with $\pi_x\sigma_n = I_\varepsilon$. Define
\[
B_{\sigma_n, k} = \{q \in \sigma_n : g^i\pi_z f(q) \not \in I^- \mbox{ for } i=0, 1, 2, \cdots, k\}\ ,
\]
 the set of points that remain in the ``bad region" $k$ times under $g$.  We claim that there exist $C_1>1$ independent of $\varepsilon$ such that $m_{\pi_{xy}\sigma_n}(\pi_{xy}B_{\sigma_n, k})
< (C_1\varepsilon) c_1^k$.
 
 For $n=0$, $\sigma_0 \subset \gamma_0$ is a segment with $\pi_x\sigma_0 = I_\varepsilon$. We show that  $m_{\sigma_0}(B_{\sigma_0, k}) < (C_1\varepsilon) c_1^k$ where 
$m_{\sigma_0}$ is Lebesgue measure on $\sigma_0$ and $c_1<1$ is as at the
beginning of Sect. \ref{Sec:ResStrong}.
To see this, observe  that (i) $m(\pi_z f(B_{\sigma_0, k})) \le c_1^k$;
(ii) on $I^+$ the density of $(\pi_z \circ f)_*m_{\sigma_0}$
with respect to Lebesgue measure $m$ on $\T$
is $\mathcal O(\varepsilon)$, because $\partial_u^Ar > \frac12 \varepsilon^{-1}$ (see (B3));
and (iii) under iterates of $g$, points that leave $I^+$ enter $I^-$ a finite number of 
steps later.

For a general $n$, the same argument as for the case $n=0$ works, provided the condition 
$$
(*)  \hskip 0.5in \partial_u^A\psi_n > \frac12 \varepsilon^{-1} \quad \mbox{on} \quad \psi_n^{-1}(I^+)\ 
$$
is satisfied. Assuming that for now, we note that $\pi_{xy} \gamma_n$ has length 
$(\lambda^u)^n$ times that of $\gamma_0$ where $\lambda^u$ is the expansion constant of $A$,
so it crosses the strip $I_\varepsilon \times \T$  
$ \le \mbox{const} \ (\lambda^u)^n$ times.
At the same time, $m_{\pi_{xy}\sigma_n}=(\lambda^u)^{-n}A^n_*m_{\sigma_0}$,
so summing $m_{\pi_{xy}\sigma_n}(\pi_{xy}B_{\sigma_n, k})$ over all the components $\sigma_n$
of $\gamma_n$, one obtains for each $n$ the same estimate (up to a uniform factor)
as in the $0$th step.

To estimate $g'$ at the $n$th step, let
\[
{\bf B}_n := \{p \in \gamma_0: \pi_z f(F^n(p)) \not  \in I^-\}\ .\]
Assuming (*), and adding up all the points that fell into the bad set in any one of the previous 
 iterates and have remained there up until time $n$, we have  
\[
m_{\gamma_0}({\bf B}_n)  \le \mbox{const}   \sum_{0 \le k \le n}  (C_1 \varepsilon) c_1^{n-k}
\le \mbox{const} \  (C_1 \varepsilon) (1-c_1)^{-1}\ .
\]
In this estimate, we have overcounted in the following way: For $q \in \sigma_i$ for some $i$,
i.e., $\pi_x(q) \in I_\varepsilon$, if it remains in the ``bad set" for $j$ iterates and $F^j q \in \sigma_{i+j}$, 
then the estimate in step $i+j$ takes over and what happened at step $i$ becomes moot 
but we have continued to count it in the estimate above.

This last estimate implies that for any limit point $\mu$ of 
$\{\frac1n \sum_{i=0}^{n-1} F^i_*\mu_0\}_{n=1,2,\cdots}$, 
we have $g' \le c_0 < 1$ for a fraction $1-\mathcal O(\varepsilon)$ of the mass. Since
$g'$ is bounded, the claim is proved for $\varepsilon$ sufficiently small.

It remains to prove (*), and we will do that inductively. Assume (*) holds at step $n-1$. We consider
$\gamma_n$, and let $p \in \gamma_n$ be such that $\psi_n(\pi_{xy}p) \in I^+$ . There are
several possibilities for how that could have come about:

\medskip \noindent
Case 1. $\pi_x(p) \not \in I_\varepsilon$. In this case, $\theta_n(\pi_{xy}p) = \psi_n(\pi_{xy}p)
\in I^+$, which implies $\pi_z f (F^{-1}p) \in I^+$.
Condition (B1) ensures that
$\partial_u^A\psi_n(\pi_{xy}p) = \partial_u^A\theta_n(\pi_{xy}p) > \partial_u^A\psi_{n-1}(\pi_{xy} F^{-1}p)$, 
which by hypothesis is $>\frac12 \varepsilon^{-1}$.

\medskip \noindent
Case 2. $\pi_x(p) \in I_\varepsilon$ and $\theta_n(\pi_{xy}p) \in g(I^+)$. This also implies 
$\pi_z f F^{-1}p \in I^+$. For the same reason as above, $\partial_u^A\theta_n(\pi_{xy}p) > \frac12 \varepsilon^{-1}$, so $\partial_u^A\psi_n(\pi_{xy}p)$ can only be larger.

\medskip \noindent
Case 3. $\pi_x(p) \in I_\varepsilon$ and $\theta_n(\pi_{xy}p) \not \in g(I^+)$. 
Since $\psi_{n-1}(\pi_{xy}F^{-1}p) \not \in I^+$, we are guaranteed only that 
$\partial_u^A \theta_n(\pi_{xy}p) \ge 0$.
However, from $\theta_n(\pi_{xy}p) \not \in g(I^+)$ and $\psi_n(\pi_{xy}p) \in I^+$, 
it follows that $r(\pi_{xy}p) \not \in (-d, d)$ (see (B3) for the definition of $d$). Hence
$\partial_u^A \psi_n(\pi_{xy}p) \ge \partial_u^A r(\pi_{xy}p) > \frac12 \varepsilon^{-1}$ by (B3).

\medskip
This completes the proof of Item 1.

\bigskip \noindent
{\it 2. Existence of SRB measure.} Let $\mu_\varepsilon$ be as constructed. Since 
$\mu_\varepsilon(\T^2 \times I^-) = 1- \mathcal O(\varepsilon)$,
there is at least one ergodic component $\nu_\varepsilon$ of $\mu_\varepsilon$ with 
$\nu_\varepsilon(\T^2 \times I^-) = 1- \mathcal O(\varepsilon)$. It follows that
for $\varepsilon$ small enough, the Lyapunov exponent in the $z$-direction is strictly 
negative $\nu_\varepsilon$-a.e.  We claim that $(\pi_{xy})_*\nu_\varepsilon =\mu^A$
where $\mu^A$ is the SRB measure of the Anosov map $A$. This is because
$(\pi_{xy})_* (\frac1n \sum_0^{n-1} F^i_*\mu_0)$ converges to $\mu^A$, 
so $(\pi_{xy})_*\mu_\varepsilon=\mu^A$, and since $\mu^A$ is ergodic, we have 
$(\pi_{xy})_*\nu_\varepsilon = (\pi_{xy})_*\mu_\varepsilon$. 

We do not attempt the usual construction of SRB measures used for Axiom A attractors or 
in Section \ref{Sec:RegMonotStatist} due to technical issues with distortion estimates on 
unstable manifolds. 
Instead, we prove the SRB property of $\nu_\varepsilon$ by appealing to a result of \cite{Tsujii}.
Theorem C of \cite{Tsujii} asserts that an invariant probability measure $\nu$ with no zero Lyapunov
exponents is SRB if the following holds:

\bigskip
(*) \quad 
{\it for any Borel set $X$ with $\nu(X)>0$, the strongly stable set of $X$, 
\[
\Gamma^s(X):=\{y\in M:\; \limsup_{n\rightarrow\infty} \log d(f^nx,f^ny)<0 \mbox{ for some }x\in X \},
\] 

\qquad has positive Lebesgue measure.}

\bigskip
We verify condition (*) as follows: Fix $X$ with $\nu_\varepsilon(X)>0$. 
Shrinking $X$, we may assume it consists of $\nu_\varepsilon$-typical points,
so that at every $p \in X$, 
there is a 2-dimensional local stable manifold contained in 
$ W_A^s(\pi_{xy}p) \times \T$ where $ W^s_A(\cdot)$ is the stable manifold
of $A$. These local stable disks are contained in $\Gamma^s(X)$.

Let $\mathcal F$ be the 2-dimensional foliation of $\T^2 \times \T$
whose leaves are $ W_A^s \times \T$.
We fix a stack $\mathcal S$ of $\mathcal F$-disks with $\nu_\varepsilon(X \cap \mathcal S)>0$, 
i.e., we fix a local unstable manifold $W$ of $A$ 
and a collection of disks  $\{\mathcal F_\alpha, \alpha \in W\}$ where 
$\mathcal F_\alpha$ 
is a disk of radius $r$ in an $\mathcal F$-leaf centered at $\alpha \in W$, and let 
$\mathcal S = \cup_{\alpha \in W} \mathcal F_\alpha$.
As $(\pi_{xy})_*\nu_\varepsilon$ is the SRB measure of $A$, by the absolute continuity
of the stable foliation of $A$, any positive $\nu_\varepsilon$-measure 
subset of $\mathcal S$ projects to a positive Lebesgue measure set on $W$.

Now let $m$ denote 3D Lebesgue measure on $\T^2 \times \T$. 
Distintegrating $m$ into a family of
regular conditional probabilities $\{m_\alpha\}$ on $\{\mathcal F_\alpha\}$ and a transverse measure
$m_T$ on $W$, we have that for any Borel set $B \subset \T^2 \times \T$, 
\begin{equation} \label{disintegration}
m(B) = \int m_\alpha(B \cap \mathcal F_\alpha) dm_T(\alpha)\ .
\end{equation}
By the absolute continuity of the $\mathcal F$-foliation, which is an immediate
consequence of the absolute continuity of the stable foliation for Anosov maps, we have that
$m_T$ is equivalent to Lebesgue measure on $W$. As $m_\alpha(\Gamma^s(X))>0$
for a positive $m_T$-measure set of $\alpha$, we conclude that $m(\Gamma^s(X))>0$.

\bigskip \noindent
{\it 3. The $\varepsilon \to 0$ limit.} Choosing $I^-$ arbitrarily small, we have  
$\mu_\varepsilon(\T^2 \times I^-) = 1- \mathcal O(\varepsilon)$ as $\varepsilon \to 0$
though the constants depend on the size of $I^-$. 
The assertion follows.
\end{proof}

 Examples of coupled systems with rare interactions in the literature include
the coupled maps lattices of \cite{keller2009map}, and statistical mechanics models such as 
\cite{bunimovich1992ergodic}, \cite{balint2017limiting} and \cite{eckmann2006nonequilibrium}.


\section{(Straightforward) Generalizations}\label{Sec:Generalizations}

For conceptual clarity, we have chosen to present our results 
for the coupling of two specific systems:
an Anosov diffeomorphism of $\T^2$ and a circle map with a sink and a source.
We now discuss generalizations that are either already known, or whose proofs 
require only minor, nonsubstantive modifications of those  
 in Sections \ref{Sec:SmallInte}-\ref{Sec:RareStrongInterac}.

\bigskip 
Theorem \ref{Thm:SmallPert} is valid for any diffeomorphism $F$ of
a manifold of any dimension with a uniformly hyperbolic attractor and the fiber map $g$ 
can be any diffeomorhphism with a sink -- except that the basin of the attractor
 need not have full Lebesgue measure. This is not new \cite{robinson1976structural, ruelle1976measure, young1990large}.
  In the case where $A$ is an Anosov diffeomorphisms, 
the persistence of the attractor as an invariant manifold in the normally hyperbolic case is also well known \cite{hirsch1977invariant}.

\bigskip
The results of Sections \ref{Sect:RegMonotINtGeom}-\ref{Sec:RareStrongInterac} are easily extended to the following situations: The base map $A$ 
can be a uniformly expanding circle map (working with  inverse limits); or it can be an Anosov diffeomorphism or any diffeomorphism with an Axiom A attractor.
The dimension of the base manifold is irrelevant, but our proofs as given rely on
the fact that $\dim(E^u)=1$, and  it is crucially important that $r$ increases 
monotonically along unstable curves. 

An example that is especially interesting geometrically
is when the base map $A: \mathbb D_2 \times \T \circlearrowleft$ (where $\mathbb D_2$ is a 
2-dimensional disk) has a solenoidal attractor, and $r$ increases monotonically as one
winds around the $\T$-component of the base. Here we obtain an attractor in the 
4-dimensional phase space $(\T^2 \times \T) \times \T$.
Under the conditions of Sections \ref{Sect:RegMonotINtGeom} and 5, we have a ``double-frequency 
solenoid" whose $W^u$-curves wind around the two-torus in its $3$rd and $4$th
dimensions in a fairly regular manner. Under the assumptions of Section \ref{Sec:RareStrongInterac}, 
$W^u$-curves wind around the 3rd dimension in 
a regular fashion making abrupt excursions around the 4th.
In both cases, the attractor so obtained is 
not uniformly hyperbolic as it has mixed behavior in the 4th dimension.

\bibliographystyle{abbrv}

\bibliography{Untitled}

\bigskip

\noindent
\emph{E-mail address:} matteo.tanzi@nyu.edu

\noindent
\emph{E-mail address:} lsy@cims.nyu.edu

\end{document}